\documentclass[leqno, letterpaper]{amsart}
\usepackage{ marvosym }
\usepackage{amssymb, mathtools, amsmath}
\usepackage{tikz, tikz-cd, graphicx}
\usepackage[normalem]{ulem}
\usepackage{color}

\usepackage{hyperref}

\usetikzlibrary{matrix,arrows,decorations.pathmorphing}

\usepackage{hyperref}
\usepackage{scalerel,stackengine}
\stackMath
\newcommand\reallywidehat[1]{%
\savestack{\tmpbox}{\stretchto{%
  \scaleto{%
    \scalerel*[\widthof{\ensuremath{#1}}]{\kern-.6pt\bigwedge\kern-.6pt}%
    {\rule[-\textheight/2]{1ex}{\textheight}}
  }{\textheight}%
}{0.5ex}}%
\stackon[1pt]{#1}{\tmpbox}%
}
\newcommand{\Addresses}{{
  \bigskip
  \footnotesize

  Damjan PI\v{S}TALO, E-mail: \textit{damjan.pistalo@gmail.com}, MathSciNet Author ID:1292911
}}

\usepackage{thmtools}
\usepackage{thm-restate}

\usepackage{amscd}
\usepackage{amsmath}
\usepackage{amsthm}
\usepackage{mathrsfs}
\usepackage{latexsym}
\usepackage{amssymb}
\usepackage{amsfonts}
\theoremstyle{plain}

\usepackage{tikz}
\usetikzlibrary{arrows,chains,matrix,positioning,scopes,snakes}

\tikzset{>=stealth',every on chain/.append style={join},
         every join/.style={->}}
\tikzset{
    >=stealth',
    punkt/.style={
           rectangle,
           rounded corners,
           draw=black, very thick,
           text width=6.5em,
           minimum height=2em,
           text centered},
    pil/.style={
           ->,
           thick,
           shorten <=2pt,
           shorten >=2pt,}
}

\usepackage[T1]{fontenc}
\usepackage{xpatch}  
\xpatchcmd{\subequations}{\alph{equation}}{-\arabic{equation}}{}{}

\newcommand{\op}[1]{\!\!\mathop{\rm ~#1}\nolimits}

\newtheorem{theo}{Theorem}
\newtheorem{rem}[theo]{Remark}
\newtheorem{prop}[theo]{Proposition}

\newtheorem{cor}[theo]{Corollary}

\newtheorem{defi}[theo]{Definition}

\newtheorem{theorem}[theo]{Theorem}
\newtheorem{lemma}[theo]{Lemma}

\usepackage{stackrel}
\usepackage{epigraph}

\usepackage{lmodern}
\usepackage{here}
\usepackage{url}
\usepackage{hyperref}

\usepackage{hyperref}
\usepackage{tikz-cd}

\allowdisplaybreaks

\newsavebox{\pullback}
\sbox\pullback{%
\begin{tikzpicture}%
\draw (0,0) -- (1ex,0ex);%
\draw (1ex,0ex) -- (1ex,1ex);%
\end{tikzpicture}}

   \numberwithin{equation}{section}

\tikzcdset{scale cd/.style={every label/.append style={scale=#1},
    cells={nodes={scale=#1}}}}

\title{Homotopical algebra of Lie-Rinehart pairs}
\author{Damjan Pi\v{s}talo}
\date{\today}
\begin{document}

\maketitle

\begin{abstract}

Dwyer-Kan localization at pairs of quasi-isomorphisms of the category of dg Lie-Rinehart pairs $(A,M)$, where $A$ is a semi-free cdga over a field $k$ of characteristic zero and $M$ a cell complex in $A$-modules, is shown to be equivalent to that of strong homotopy Lie-Rinehart (SH LR) pairs satisfying the same cofibrancy condition. Latter is a category of fibrant objects. We introduce cofibrations of SH LR pairs, construct factorizations, and prove lifting properties. Applying them, we show uniqueness up to homotopy of certain BV-type resolutions. Restricting to dg LR pairs whose underlying cdga is of finite type, and using a different (co)fibrancy condition, we show that the functor $(A,M)\mapsto A$ is a Cartesian fibration with presentable fibers. The two (co)fibrancy conditions yield equivalent $\infty$-categories under Dwyer-Kan localization.
\end{abstract}

\tableofcontents

\section{Introduction}
Lie algebroids have found a great number of applications in different branches of geometry and physics, such as foliation theory, Poisson geometry, field theory, and beyond. It is hence expected that a solid homotopy theory for Lie algebroids would be a useful tool in homotopical geometry. 

In the algebraic setting, one speaks of (sheaves of) dg Lie-Rinehart pairs.
To date, homotopy theory of such objects has been developed mostly with morphisms over a fixed derived (affine) scheme (\cite{Nui1}), which is a context suitable for applications in derived deformation theory (\cite{Nui2}). For general morphisms of dg LR pairs, not much has been done systematically.

A morphism of dg Lie-Rinehart pairs $(f_0,f):(A,M)\to (B,N)$ is a morphism of cdga-s $f_0:B\to A$ together with a morphism of $A$-modules $f:M\to A\otimes_B N$, compatible with anchors and brackets. It factors uniquely through the Lie-Rinehart pullback $(A,f_0^!N)=(A, \op{Der}(A)\times_{\op{Der(A,B)}}B\otimes_A N)$. With this, under a finiteness condition on the underlying cdga-s, and with appropriate (co)fibrancy conditions, the functor which associates to a dg LR pair its underlying affine derived scheme (opposite cdga) is a Cartesian fibration (Theorem \ref{cartfib}).  In the absence of a "pushforward", the functor is not a Cartesian opfibration, hence there is no particular reason to expect the category of dg Lie-Rinehart pairs be a presentable $\infty$-category.  The aim of this paper is to show that, apart from the existence of (homotopy) colimits, much of abstract homotopy theory still applies. In fact, most of the results are proved without the above finiteness condition.

Throughout the paper, $k$ is a fixed field of characteristic zero. The category of dg LR pairs $(A,M)$, with $A$ semi-free cdga and  $M$ a cell complex in $A$-modules, localized at pairs $(f_0,f)$ of quasi-isomorphisms, is denoted by ${\tt dgLR}(k)^{\op{cof}}$. It is proven to be equivalent to the localization of the category of strong homotopy Lie-Rinehart pairs which satisfy the same cofibrancy condition. Latter category is the Dwyear-Kan localization of a category of fibrant objects (\cite{P}). Worth noting is that we distinguish between two generalities of strong homotopy morphisms: \emph{linear}, and  \emph{general}. Although both notions give categories of fibrant objects, there is a difference in regards to the lifting properties (Lemma \ref{lemma19}, Lemma \ref{lemma21}). By fiddling with Koszul duality for Lie algebras, we establish in Theorem \ref{theo1} the equivalence
$$\mathcal{L}_{-}C_{-}:{\tt SHLR}^{\op{cof,lin}}(k)\rightleftarrows {\tt dgLR}(k)^{\op{cof}}:\imath$$ between the infinity categories of dg LR pairs and SH LR pairs with linear SH morphisms. Similar constructions in different contexts appear in \cite{Nui1}, and \cite{CCN}. In Theorem \ref{theo2}, we show that the inclusion of a wide subcategory
$$ {\tt SHLR}^{\op{cof,lin}}(k)\hookrightarrow {\tt SHLR}^{\op{cof}}(k)$$ is also an equivalence of infinity categories, by explicitly constructing its right adjoint, for which adjunction's unit and counit are evidently weak equivalences.  It should be noted that Theorem \ref{theo1} is valid in the more general setting where $A$ is any cdga. However, in this situation, it is unclear if the equivalence extends to general morphisms. On the level of 1-categorical Gabriel-Zisman localization, the equivalent categories of Theorems \ref{theo1} and \ref{theo2} are further equivalent to the category of SH LR pairs $(A,M)$ with general $A$ (and $M$ a cell complex of $A$-modules), and general morphisms (Corollary \ref{GZ}).

In section \ref{cart}, we consider dg Lie-Rinehart pairs whose underlying cdga-s are of finite type.
Among them, we distinguish \emph{nice} dg LR pairs, which are nice in the sense that the pullback of a nice dg LR pair is both of the correct homotopy type, and a nice dg LR pair itself. On one side, the category of nice dg LR pairs is equivalent (under the Dwyer-Kan localization) to ${\tt dgLR}^{\op{ft,cof}}(k)$, and on the other side,
$$p:{\tt dgLR}^{\op{ft,nice}}(k)\to {\tt cdga}^{\op{ft}}(k)^{\op{op}},\hspace{10pt}(A,M)\mapsto A$$
is a cartesian fibration. Its fibers are presentable infinity categories studied in \cite{Nui1}. 

Sections \ref{lift} and \ref{BV} do not have the same infinity categorical flavor, focusing instead on "classical" computational tools in abstract homotopy theory, beyond what is given by the structure of a category of fibrant objects. 

In section \ref{lift},  after defining cofibrations of SH LR pairs, we prove lifting properties and establish tools such as factorizations and lifting properties. This is the first instance where the difference between linear and general morphisms comes into play: In the lifting diagrams with linear morphisms over a fixed cdga, the lift is also a linear morphism. In general, the lift needs not be linear, even if all the maps in the lifting diagram are.

In section \ref{BV}, we deal with resolutions of dg LR pairs $(A,M)$ by SH LR pairs $(QA,QM)$ with $QA$ semi-free, and $QM$ a cell complex in $QA$ modules. Here, the machinery of section \ref{lift} is used to prove the usual uniqueness-up-to-homotopy. The section is largely motivated by the B(F)V-BRST complex, and builds on homotopical-algebraic interpretation of Koszul-Tate resolutions studied in \cite{PP}.

The remaining sections serve to introduce the 1-categorical notions used throughout the paper.

In section \ref{one}, we recall the notions of dg LR pairs and their morphisms, and we review the free–forgetful adjunction to anchored modules.

Section \ref{two} deals with SH LR pairs and their (linear and general) morphisms. In the literature  (\cite{Vit}, \cite{Hueb}), an SH LR pair $(A,M)$ is commonly defined in terms of a multiderivative on $M[1]$. Equivalently, it is a dg $A$-module $M$ together with an $L_{\infty}$-structure subjected to a certain Leibniz rule, and an $A$-multilinear $L_{\infty}$ morphism (anchor) $\Gamma:M\rightsquigarrow \op{Der}(A)$. Definition of a SH morphism $(A,M)\rightsquigarrow(B,N)$, applicable to SH LR pairs $(A,M)$ with $M$ a dualizable $A$-module, is that it is opposite to a morphism of the Chevalley-Eilemberg complexes $f^*:(\widehat{\op{Sym}}_B(N^*[-1]),d_{\op{CE}})\to (\widehat{\op{Sym}}_A(M^*[-1]),d_{\op{CE}})$, whose composition with the augmentation $\widehat{\op{Sym}}_A(M^*)\to A$  factors through the augmentation $\widehat{\op{Sym}}_B(M^*)\to B$. Applying the language of differential graded pro-algebras developed in \cite{P}, the definition extends straightforwardly from dualizable to cofibrant, and even graded-flat modules. These morphisms are referred to as \emph{general}. \emph{Linear} SH morphisms additionally satisfy the condition $f^*(A)\subseteq B$. Characterization of linear SH morphisms in terms of SH LR structure is provided in the Definition \ref{defLin}. Over a fixed base cdga $A$, they are $A$-multilinear morphisms of $L_\infty$-algebras, which respect the anchors.

\subsection*{Notation}

We adopt cohomological convention: differentials raise degree by 1.

The category of non-positively graded differential  graded-commutative algebras over $k$ is denoted by ${\tt cdga}(k).$ Given $A\in{\tt cdga}(k)$, ${\tt Mod}(A)$ denotes the category of unbounded cochain complexes of $k$-vector spaces with an $A$-action. Explicitly, a cochain complex $M$ is an $A$-module if it is a graded $A$-module, and its differential satisfies the graded Leibniz identity
\begin{equation}\label{leib1}d(a\cdot m)=d_Aa\cdot m+(-1)^{|a|}a\cdot dm.\end{equation}
Morphisms are required to preserve the $A$-module structure and respect the differential.

The free graded symmetric algebra of an $A$-module $ M$ is 
$$\op{Sym}_{ A}M=\oplus_{k\in\mathbb{N}}{ M}^{\otimes_{ A} k}/{\mathcal{I}},$$ where $\mathcal{I}$ is the ideal generated by $\{m\otimes n- (-1)^{|m|\cdot|n|} n\otimes m:m,n\in M\}$.
Its multiplication (the graded symmetric tensor product) is denoted by $\odot$.  The index $k$ is called weight. $A$-module of weight $k$ elements in $ \op{Sym}_{A}M$ is denoted by $\op{Sym}_{ A}^kM$. Sometimes, the symmetric product  $m_1\odot\cdots\odot m_k$ is denoted by $m_{1\ldots k}$.

The group of all the permutations of a set with $k$ elements is denoted by $\Sigma_k$. The set of $(l,m)$-unshuffles is denoted by $\op{Sh}(l,m)\subset \Sigma_{l+m}$. 

Given a $k$-permutation $\sigma\in \Sigma_k$ together with $m_1,\ldots, m_k\in M$, the number $|m_1,\ldots,m_k|_{\sigma}\in \{\pm 1\}$ is implicitly defined by
$$m_{\sigma(1)}\odot\ldots\odot m_{\sigma(k)}=(-1)^{|m_1,\ldots,m_k|_{\sigma}}m_{1}\odot\ldots\odot m_{1}.$$

$\mathbb{N}$ denotes the set of natural numbers including zero. The set of natural numbers without zero is denoted by $\mathbb{N}_{>0}.$
\subsection*{Acknowledgments} The author gratefully acknowledges that the majority of this work was conducted during a postdoctoral fellowship at the University of Luxembourg.

\section{The category of dg Lie-Rinehart pairs}\label{one}
To begin with, we recall the notions of anchored modules and LR pairs, and establish certain preliminary results in somewhat greater generality then in the available literature, but by no means unexpected. Of particular importance is the free forgetful adjunction between anchored modules and LR pairs, which simultaneously extends \cite[Theorem 2.1.2]{Kap} to a variable base algebra and to the (d)g setting.

\begin{defi}
\emph{(Differential) graded anchored module} is a pair $(A,M)$, for $A\in{\tt c(d)ga}(k)$, $M\in {\tt Mod}(A)$, together with a morphism of $A$-modules $\Gamma:M\to\op{Der}(A)$ called anchor.

\emph{A morphism of (differential) graded anchored modules} $(f_0,f):(A,M)\to (B,N)$ is a morphisms $f_0:B\to A$ of (differential) graded algebras together with a morphism of $A$-modules $f:M\to A\otimes_B N$ such that the diagram
\begin{center}
\begin{tikzcd}[column sep=large]
M\arrow[r,"f"]\arrow[d]&\arrow[d]A\otimes_B N\\
\op{Der}(A)\arrow[r,"-\circ f_0"]&\op{Der}(B,A).
\end{tikzcd}
\end{center}
commutes. Vertical map on the right is given by $a\otimes n\mapsto a\cdot f_0\circ
\Gamma(n).$

The category of (d)g anchored modules is denoted by ${\tt(d)gAnch}(k)$.
\end{defi}

\begin{defi}
\emph{A (differential) graded Lie-Rinehart pair} is a (d)g anchored module $(A,M)$ together with a (d)g Lie algebra structure on $M$ such that
\begin{itemize}
\item $\Gamma:M\to \op{Der}(A)$ is a morphism of (d)g Lie algebras; and 
\item the Lie bracket in $M$ satisfies the graded Leibniz rule
$$[m,an]=\Gamma(m)(a)\cdot n+(-1)^{|a||m|}a\cdot[m,n].$$
\end{itemize}.

\emph{A morphism of (differential) graded Lie-Rinehart pairs} $(f_0,f):(B,M)\to(A,N)$ is a morphism of anchored modules compatible with the brackets, in the sense that for $f(m)=\sum_i b_i\otimes n_i$, $f(m')=\sum_j b'_j\otimes n'_j$,
\begin{equation}\label{LRmor}f([m,m'])=\sum_{i,j}(-1)^{|b'_j||n_i|}b_ib'_j\otimes[n_i,n'_j]+\sum_j\Gamma(m)(b'_j)\otimes n'_j-(-1)^{|m||m'|}\sum_i\Gamma(m')(b_i)\otimes n_i.\end{equation}

The category of (d)g Lie-Rinehart pairs is denoted by ${\tt (d)gLR}(k).$
\end{defi}

\begin{theorem}
The forgetful functor $\op{For}:{\tt (d)gLR}(k)\to{\tt(d)gAnch}(k)$ admits a left adjoint free Lie-Rinehart functor, denoted by $\op{LR}$.
\end{theorem}
\begin{proof}

Given a (differential) graded vector space $V$, the free (d)g Lie $k$-algebra over $V$, denoted by $\op{Lie}(V)$,  has a natural weight grading
$$\op{Lie}(V)=\bigoplus_{k\geq 1}\op{Lie}^k(V),$$
where $\op{Lie}^k(V)$ is spanned by the brackets involving exactly $k$ elements of $V$.

Consider now an  anchored (d)g module $(A,M)$. If the anchor is trivial, the free Lie-Rinehart pair $(A,\op{LR}_A(M))$ is the free (d)g Lie $A$-algebra over $A$, hence naturally graded. In general, $\op{LR}_A(M)$ is the union of a filtered anchored (d)g $A$-module
\begin{equation}\label{filt1}\op{LR}_A^{\leq 1}(M)\subset \op{LR}_A^{\leq 2}(M)\subset\ldots \end{equation}
defined recursively starting from $\op{LR}_A^{\leq 1}(M)=M$, such that for each $k\geq 1$, $\op{LR}_A^{\leq 2}(M)$ is, as a (d)g vector space, a quotient of
$\op{Lie}^{\leq k}(M)=\bigoplus_{l=1\ldots k}\op{Lie}^l(M).$
Denote by $q_k$ the quotient map.

Suppose that the anchored $A$-module $\op{LR}_A^{\leq k}(M)$ has been defined for some $k$.
$\op{LR}_A^{\leq k+1}(M)$ is the quotient of $\op{Lie}^{\leq k+1}(M)$ by the following relations:
\begin{equation}\label{rel1}
[x,r]=0,\hspace{5pt}x\in \op{LR}^{\leq l}(M),r\in\op{Ker}(q_m),l+m=k+1;
\end{equation}
\begin{equation}\label{rel2}
[ax,y]-(-1)^{|a||x|}[x,ay]=(-1)^{(|a|+|x|)|y|}\Gamma(y)(a)\cdot x-(-1)^{|a||x|}\Gamma(x)(a)\cdot y,
\end{equation}
for $x\in\op{LR}^{\leq l}(M)$, $y\in\op{LR}^{\leq m}(M)$,  and  $l+m=k+1.$
Anchor is extended by
$$\Gamma([x,y])=[\Gamma(x),\Gamma(y)], \hspace{5pt}x\in\op{LR}^{\leq l}(M),y\in\op{LR}^{\leq m}(M), l+m=k+1,$$
and the $A$-module structure by
$$a\cdot[x,y]=(-1)^{|a||x|}([x,ay]-\Gamma(x)(a)\cdot y), \hspace{5pt}a\in A,x\in\op{LR}^{\leq l}(M),y\in\op{LR}^{\leq m}(M).$$

A map of (d)g anchored modules $(f_0,f):(A,M)\to (B,N)$ induces the map of (d)g LR pairs $(f_0,\op{LR}(f)):(A,\op{LR}_A(M))\to (B,\op{LR}_B(N))$, with $\op{LR}(f)$ defined recursively by it restrictions (which are morphisms of anchored modules)
$$\op{LR}^{\leq k}(f):\op{LR}_A^{\leq k}(M)\to A\otimes_B\op{LR}_B^{\leq k}(N),$$ as follows
\begin{itemize}
\item $\op{LR}^{\leq 1}(f)=f$;
\item Assuming $\op{LR}^{\leq k}(f)$ is defined for some $k$,  $\op{LR}^{\leq k+1}(f)$ is defined by the Equation \ref{LRmor}, for $m\in\op{LR}^{\leq p}(M),$ $m' \in\op{LR}^{\leq q}(M),$ with $p+q=k+1$.
\end{itemize}
\end{proof}

Final observation in the present section is that dg LR pairs can be pulled back just as their classical counterparts.

Given a dg LR pair $(A,M)$, and a morphism $f_0:A\to B$ of cdga-s, pullback of $(A,M)$ along $f_0$ is denoted by $(B,f_0^!M)$. As a dg $B$-module, $f_0^!M$ is equal to the pullback $$\op{Der}(B,B)\times_{\op{Der}(A,B)}B\otimes_AM.$$ Anchor is the projection to $\op{Der}(B,B)$, and the bracket is given by
\begin{equation*}
\begin{split}[(\rho, &\sum_i b_i\otimes m_i), (\rho', \sum_j c_j\otimes n_j)]\\=&([\rho,\rho'],\sum_{i,j}(-1)^{|m_i||c_j|}b_ic_j\otimes[m_i,n_j]+\sum_j\rho(c_j)\otimes n_j-(-1)^{|\rho'|(|b_i|+|m_i|)}\sum_i\rho'(b_i)\otimes m_i).\end{split}\end{equation*}
Denote by $$p:\op{Der}(B,B)\times_{\op{Der}(A,B)}B\otimes_AM\to B\otimes_A M$$
the canonical map. Then
$$(f_0,p):(B,f_0^!M)\to (A,M)$$ is a morphism of dg LR pairs

We state the factorization property without the proof. It is a special case of the Proposition \ref{fact}.
\begin{prop}\label{fact1}
Given a morphism of dg LR pairs $(f_0,f):(B,M)\to (A,N)$ there exists a unique morphism $(\op{Id}_{B},\overline{f}):(B,M)\to (B,f_0^!N)$ for which $$(f_0,f)=(f_0,p)\circ (\op{Id}_{B},\overline{f}).$$
\end{prop}

\section{The category of strong homotopy Lie-Rinehart pairs}\label{two}
In the present section we introduce the notions of SH LR pairs and SH morphisms, and we prove certain preliminary results used throughout the paper. We begin by recalling the relevant constructions concerning $L_\infty$-algebras.

Given an $L_\infty$-algebra $L$ over $k$, denote by $C_k(L)$ its reduced Chevalley-Eilenberg coalgebra -- kernel of the CE coalgebra's counit. Explicitly, it is the reduced graded-symmetric coalgebra $\op{Sym}_k^{>0}(L[1])$ for the coproduct
$$\overline{\Delta}(l_1\odot\cdots\odot l_n)=\sum_{\substack{k=1,\ldots n-1\\ \sigma\in \op{Sh}(k,n-k)}}(-1)^{|l_1,\ldots,l_n|_\sigma}(l_{\sigma(1)}\odot\cdots\odot l_{\sigma(k)})\otimes (l_{\sigma(k+1)}\odot\cdots\odot l_{\sigma(n)}),$$ together with a differential cogenerated (in the sense of \cite[Corollary 11.5.9]{Manetti}) by the negative d\'{e}calage of $n$-ary brackets
\begin{equation}\label{infty_shift}D_{C_k(L)}^n(l_1\odot\cdots\odot l_n)=-(-1)^{\sum_{i=1}^{n}(n-i)|l_i|}s^{-1}[s l_1,\ldots,s l_n]_n,\hspace{10pt}n\in\mathbb{N}_{>0}.\end{equation} 

$L_\infty$ structure is fully encoded within the Chevalley-Eilenberg differential. In fact, given a graded vector space $V$, differential on the graded symmetric coalgebra $\op{Sym}_k(V[1])$ is commonly referred to as an $L_\infty[1]$-algebra structure on $V[1]$. Equation \ref{infty_shift} provides 1-1 correspondence between $L_{\infty}$-structures on $V$, and $L_{\infty}[1]$-structures on $V[1]$, assigning to a given $L_{\infty}$-algebra its Chevalley-Eilenberg coalgebra. $L_\infty[1]$- morphisms, denoted by $V[1] \rightsquigarrow W[1]$,  are commonly defined either as morphisms between the respective Chevalley-Eilenberg coalgebras $F:\op{Sym}_k(V[1])\to \op{Sym}_k(W[1])$, or as their cogenerators
$(f^n:\op{Sym}^n_k(V[1])\to W[1])_{n\geq 1}$. Explicitly, \begin{equation}\label{extension}F(m_1\odot\cdots\odot m_n)=\sum_{\substack{\substack{r=1,\ldots n\\1\leq n_1<\ldots<n_r\leq n\\ \sigma\in\op{Sh}(n_1,\ldots,n_r)}}}(-1)^{|v_1,\ldots, v_n|_\sigma}f^{n_1}(v_{\sigma(1)\ldots \sigma(n_1)})\odot\cdots\odot f^{n_r}(v_{\sigma(n-n_r+1)\ldots\sigma(n)}).\end{equation} A map $(f^n)_n$ is called an infinity quasi-isomorphism if $f^1$ is a quasi-isomorphism. Decalage of an $L_\infty[1]$ morphism is sometimes referred to as an $L_\infty$-morphism. In particular, given a morphisms $f:L\to L'$ of dg Lie algebras, 
$$\op{d\'{e}c}^{-1}(f):L[1]\to L'[1],\hspace{10pt} l\mapsto s^{-1} f(s l)$$  (together with zero maps $\op{Sym}^k (L[1])\to L'[1]$ for $k>1$) is an $L_{\infty}[1]$ morphism. 

If $V$ is finite-dimensional, $L_\infty[1]$-algebra structure on $V$ is dually encoded by the Chevalley-Eilenberg algebra --
a formal power series $$\widehat{\op{Sym}}_k(V^*)=\varprojlim_{n}\op{Sym}_k(V^*)/\op{Sym}^{>n}_k(V^*),$$ together with a continuous square-zero differential. Equivalently, it is the pro-object $(\op{Sym}_k(V^*)/\op{Sym}^{>n}_k(V^*))_n$ in the category of Artinian graded algebras, still denoted by $\widehat{\op{Sym}}_k(V^*)$, together with a square-zero differential (Appendix \ref{Ap}).
 Later formulation immediately generalizes to infinite-dimensional graded vector spaces: A vector space $V$ is a direct limit of its finite dimensional subspaces
$V=\varinjlim_i V_i$. With this, $L_\infty[1]$-algebra structure on $V$ is dualized to a differential on the pro-Artinian graded $k$-algebra 
$$\widehat{\op{Sym}}_k(V^*):=(\op{Sym}_k(V^*_i)/\op{Sym}^{>n}_k(V^*_i))_{i,n},$$ endowing it with the structure of a pro-Artinian dg $k$-algebra.  In \cite{Pri}, the category of  pro-Artinian dg $k$-algebras is endowed with a model structure such that the above duality provides an equivalence between the category of $L_\infty$-algebras, and the opposite to the category of cofibrant pro-Artinian dg $k$-algebras.

Contrary to $L{_\infty}[1]$ algebra structures, a SH LR[1] (or $LR_\infty[1]$) structure on $(A,M)$ is not in general determined by a differential on the graded coalgebra  $\op{Sym}_A(M)$. Still, it is determined by a differential on the dual graded pro-algebra, as long as $M$ is a flat graded $A$-module.  Namely, by Lazard's theorem, a graded $A$-module $M$ is flat if and only if it is the colimit of a directed system of finite free graded $A$-modules $M=\varinjlim_{ i}M_i$. With this, a SH LR[1] structure is equivalently encoded by a square-zero differential on the pro-finite $A$-extension 
$$\widehat{\op{Sym}}_A(M^*):=( \op{Sym}_A(M^*_i)/\op{Sym}^{>n}_A(M^*_i))_{i,n},$$ 
which restricts to the kernel of the augmentation $\widehat{\op{Sym}}_A(M^*)\to A$ (\cite[Section 5]{P}, Appendix \ref{Ap}). Its weight $k-1$ component restricted respectively to $A$ and $M^*$ yields a degree one derivation
$\sigma^{k-1}:A\to\op{Sym}_A^{k-1}(M)^*$, and a degree one morphism of pro-graded-vector spaces $d^k:M^*\to \op{Sym}_A^{k}(M)^*$
which satisfies a certain Leibniz rule with respect to the $A$-module structure and $\sigma^{k-1}$. Such pair $(d^k,\sigma^k)$ is dually a weight $k-1$ \emph{multiderivation} on $M$: a pair $(D^{k},\rho^{k-1})$, where $\rho^{k-1}:\op{Sym}_A^{k-1}(M)\to \op{Der}(A)$ is a degree one morphism of graded $A$-modules, and $D^k:\op{Sym}_k^{k}(M)\to M$ a degree one morphism of graded vector spaces such that
\begin{equation*}D^k(m_1\odot\cdots\odot a\cdot m_k)=(-1)^{|a|(1+|m_1|+\ldots+|m_k|)} a\cdot D^k(m_1\odot\cdots\odot m_k)+\rho^{k-1}(m_1\odot\cdots\odot m_{k-1})(a)\cdot m_k.\end{equation*}
 In particular,  weight zero component of the Chevalley-Eilenberg differential determines differentials on $A$ and $M$ such that $A$ is a dg algebra, and $M$ is a dg $A$-module.

Finally, the differential on $\widehat{\op{Sym}}_A(M^*)$ squares to zero if and only if the family of multiderivations $(D,\sigma):=(D_k,\sigma_k)_{k\in\mathbb{N}_{>0}}$ determines an SH LR[1] structure on $M$, defined as follows: 
\begin{defi}
Let $A\in {\tt cdga}(k),$ and  $M\in{\tt Mod}(A)$. The pair $(A,M)$, together with a family of multiderivations $(D,\sigma):=(D^k,\rho^{k-1})_{k\in\mathbb{N}_{>0}}$, such that
\begin{enumerate}
\item $D^1$ is equal to the differential on $M$, and $\rho^0$ is equal to the differential on $A$;
\item $D=(D^k)_{k\in\mathbb{N}_{>0}}$ is an $L_\infty[1]$ structure on $M$; and
\item  the Maurer-Cartan equation $\rho\circ D+\frac{1}{2}[\rho,\rho]=0$, detailed within Equation \ref{MC}, is satisfied;
\end{enumerate}
is called a SH LR[1]-pair.
\end{defi}

Following characterization of SH LR[1] structures is useful:
\begin{prop}
For $A\in {\tt cdga}(k)$ and $M\in{\tt Mod}(A)$, a SH LR[1] structure on $(A,M)$ is equivalently an $L_\infty[1]$-algebra structure on $M$ whose unary bracket is the differential on $M$, together with a morphism of $L_\infty[1]$-algebras $\Gamma:M\rightsquigarrow\op{Der}(A)[1]$ called anchor, such that 

\begin{enumerate}
\item $\Gamma^n:\op{Sym}^n_k(M)\to \op{Der}(A)[1]$ is $A$-multilinear\footnote{For a graded $A$-module $N$, $A$-module structure on $N[1]$ is given by $a\cdot sn=(-1)^{|a|}s(a\cdot n)$} for all $n\in\mathbb{N}_{>0}$; and
\item the failure of $L_\infty[1]$ structure on $M$ to be to be $A$-multilinear is governed by the graded Leibniz rule
\begin{equation*}D^k(m_1\odot\cdots\odot a\cdot m_k)=(-1)^{|a|(1+|m_1|+\ldots+|m_k|)} a\cdot D^k(m_1\odot\cdots\odot m_k)+s\Gamma^{k-1}(m_1\odot\cdots\odot m_{k-1})(a)\cdot m_k.\end{equation*}
\end{enumerate}
\end{prop}
\begin{rem}\label{anch}
By (1), the anchor endows $\op{Sym}_A(M)[-1]$ with the structure of a graded anchored $A$-module.
\end{rem}
\begin{proof}
For all $n\in\mathbb{N}_{>0}$, define degree one $A$-multilinear maps \begin{equation}\label{SHLRprop}\rho^n:\op{Sym}^n_k(M)\to \op{Der}(A),\hspace{10pt} (m_1\odot\cdots\odot m_n)\mapsto s\Gamma^n (m_1\odot\cdots\odot m_n)\end{equation}
 Denoting by $\rho^{0}$ the differential on $A$, conditions (1) and (2) are clearly satisfied if and only if $(D^n,\rho^{n-1})$ is a weight $n$-multiderivation for all $n\in\mathbb{N}$. It remains to show the Maurer-Cartan equation 
$\rho\circ D+\frac{1}{2}[\rho,\rho]=0$ is satisfied if and only if $(\Gamma^n)_{n\in\mathbb{N}_{>0}}$ is an $L_{\infty}$-morphism. This is done by direct computation (for the Koszul sines, see notation):
\begin{align}\label{MC}
\rho\circ D+&\frac{1}{2}[\rho,\rho](m_1\odot\cdots\odot m_n)\\
=&\sum_{\substack{k=1,\ldots n\\ \sigma\in \op{Sh}(k,n-k)}}(-1)^{|m_1,\ldots,m_n|_\sigma}\rho^{n-k+1}(D^k(m_{\sigma(1)\ldots\sigma(k)})\odot m_{\sigma(k+1)\ldots\sigma(n)})\nonumber\\
+&\frac{1}{2}\sum_{\substack{k=0,\ldots,n\\ \sigma\in \op{Sh}(k,n-k)}}(-1)^{|m_1,\ldots,m_n|_\sigma+|m_{\sigma(1)\ldots\sigma(k)}|}[\rho^k(m_{\sigma(1)\ldots\sigma(k)}),\rho^{n-k}(m_{\sigma(k+1)\ldots\sigma(n)})]\nonumber\\
=&s\left(\sum_{\substack{k=1,\ldots n\\ \sigma\in \op{Sh}(k,n-k)}}(-1)^{|m_1,\ldots,m_n|_\sigma}\Gamma^{n-k+1}(D^k(m_{\sigma(1)\ldots\sigma(k)})\odot m_{\sigma(k+1)\ldots\sigma(n)})\right.\nonumber\\
-&\left.\frac{1}{2}\sum_{\substack{k=1,\ldots,n-1\\ \sigma\in \op{Sh}(k,n-k)}}(-1)^{|m_1,\ldots,m_n|_\sigma}D^2(\Gamma^k(m_{\sigma(1)\ldots\sigma(k)})\odot \Gamma^{n-k}(m_{\sigma(k+1)\ldots\sigma(n)}))\right)\nonumber\\
+&d_A\circ \rho^n(m_{1\ldots n})-(-1)^{|\rho^n(m_{1\ldots n})|}\rho^n(m_{1\ldots n})\circ d_A.\nonumber
\end{align}
As
\begin{equation*}d_A\circ \rho^n(m_{1\ldots n})-(-1)^{|\rho^n(m_{1\ldots n})|}\rho^n(m_{1\ldots n})\circ d_A=d_{\op{Der}(A)}\rho^n(m_{1\ldots n})=s(-d_{\op{Der}(A)[1]}\Gamma^n(m_{1\ldots n}))
,\end{equation*} 
by \cite[Proposition 12.2.3]{Manetti}, the right hand side in \ref{MC} vanishes if and only if $\Gamma$ is an $L_\infty[1]$-morphism.
\end{proof}
\begin{rem}There is another, significantly more general notion of a SH LR pair, coming from the fact that (colored) LR operad is Koszul \cite{VDL}. In that sense, SH LR pair is a $C_\infty$-algebra $A$ together with an $L_\infty$-algebra $M$, such that $M$ is an infinity module over $A$ and vice-versa, with all the compatibility conditions up to homotopy. However, such generality is not essential for the goals of the present paper. Another difference to the fully operadic approach is that the latter gives a different (and less common) notion of morphisms. The two agree when $f_0=\op{id}_A$; for details, see Lie algebroid morphisms and comorphisms in \cite{Mac}.\end{rem}

Notion of a SH morphism is a vague concept to date, with various generalities used for various needs. When dealing graded-flat modules, one can reason in terms of Chevalley-Eilenberg algebras, and define \emph{general} morphisms of SH LR[1] pairs $(A,M)\to (B,N)$ as pairs $(f_0,f^*)$, with $f_0:B\to A$  a morphism of cdga-s, and $f^*:(\widehat{\op{Sym}}_B(N^*),d_{\op{CE}})\to(\widehat{\op{Sym}}_A(M^*),d_{\op{CE}})$ a morphism of differential graded pro-algebras, such that the diagram
\begin{equation}\label{morph}
\begin{tikzcd}
\widehat{\op{Sym}}_B(N^*)\arrow[r,"f^*"]\arrow[d]&\arrow[d]\widehat{\op{Sym}}_A(M^*)\\
B\arrow[r,"f_0"]&A
\end{tikzcd}
\end{equation}
commutes. Weight zero component of $f^*$ is dual to a morphism of $A$-modules $f^1:M\to A\otimes_B N$. A SH morphism $(f_0,f^*)$ is a weak equivalence (or SH quasi-isomorphism) if $f_0$ and $f^1$ are both quasi-isomorphisms. Observe that the notion of weak equivalence is homotopy coherent only if $A\otimes_B N$ is of correct homotopy type.

Unlike general morphisms,  \emph{linear} SH maps are defined in terms of SH LR[1] structure:

\begin{defi}\label{defLin}
Let $(A,M)$, $(B,N)$ be SH LR[1] pairs, $f_0:B\to A$ a morphism of cdga-s, and $$f=(f^n:\op{Sym}_A^n(M)\to A\otimes_BN)_{n>0}$$ a family of morphisms of graded $A$-modules. We say that $(f_0,f):(A,M)\rightsquigarrow(B,N)$ is a \emph{linear SH morphism} if
\begin{enumerate}
\item
the extension of $f$ to a map of graded coalgebras
$$F:\op{Sym}_A(M)\to \op{Sym}_A(A\otimes_BN)=A\otimes_B\op{Sym}_B(N)$$ given by Equation \ref{extension} is, under a $(-1)$-shift, a morphism of anchored modules (Remark \ref{anch}); and
\item
denoting $f^n(m_1\odot\cdots\odot m_n)=\sum_{\alpha}a^{m_{1\ldots n}}_{\alpha}\otimes n^{m_{1\ldots n}}_{\alpha}$,
\begin{equation*}
\begin{split}
&\sum_{\substack{k=1\ldots n\\ \sigma\in \op{Sh}(k,n-k)}}(-1)^{|m_1,\ldots,m_n|_\sigma} f^n(D^k(m_{\sigma(1)\ldots \sigma(k)})\odot m_{\sigma(k+1)\ldots \sigma(n)})\\=&\sum_{\substack{k=0,\ldots,n\\ \sigma\in\op{Sh}(k,n-k)}}\sum_{\alpha}(-1)^{|m_1,\ldots,m_n|_\sigma} \rho^k(m_{\sigma(1)\ldots \sigma(k)})(a_\alpha^{m_{\sigma(k+1)\ldots\sigma(n)}})\otimes n_\alpha^{m_{\sigma(k+1)\ldots \sigma(n)}}\\
+&\sum_{\substack{r=1,\ldots n\\0=n_0<\ldots<n_r= n\\ \sigma\in\op{Sh}(n_1-n_0,\ldots,n_r-n_{r-1})}}\sum_{\alpha_1,\ldots\alpha_r}(-1)^{|m_1,\ldots,m_n|_\sigma+\sum_{i=0}^{r-1} |a_{\alpha_{i+1}}^{m_{\sigma(n_{i}+1)\ldots \sigma (n_{i+1})}}|(1+ |n_{\alpha_1}^{m_{\sigma(1)\ldots \sigma (n_{1})}}|+\ldots+|n_{\alpha_{i}}^{m_{\sigma(n_{i-1}+1)\ldots \sigma (n_{i})}}|)}\\&
\frac{1}{r!}a_{\alpha_1}^{m_{\sigma(1)\ldots \sigma(n_1)}}\cdots a_{\alpha_r}^{m_{\sigma(n_{r-1}+1)\ldots\sigma(n_r)}}D^r(n_{\alpha_1}^{m_{\sigma(1)\ldots \sigma(n_1)}}\odot\cdots\odot n_{\alpha_r}^{m_{\sigma(n_{r-1}+1)\ldots\sigma(n_r)}}).
\end{split}
\end{equation*}
\end{enumerate}
Morphism $f$ is called \emph{strict} if $f^i=0$ for $i>1$.
\end{defi}
In particular, $f^1:M\to A\otimes_B N$ is a morphism of dg $A$-modules.  Assuming $A\otimes_B N$ is of the correct homotopy type, a linear map $(f_0,f)$ is a {SH quasi-isomorphism} if $f_0$ and $f^1$ are both quasi-isomorphisms.

\begin{rem}\label{rem11}For a linear SH morphism $$(\op{id}_A,f):(A,M)\rightsquigarrow(A,N),$$ the condition  $(2)$ is equivalent to the requirement that $f$ be a morphism of $L_\infty$ algebras. Indeed, all the coefficients $a^{m_{1\ldots n}}_{\alpha}$ can be set to $1$, and the condition two yields the characterization of $L_\infty$ morphisms given by \cite[Proposition 12.2.3.]{Manetti}.
\end{rem}

Equivalently, linear SH morphisms arel SH morphisms $(f_0,f^*)$ for which $f^*(B)\subseteq A$. Proof is a straightforward but tedious exercise in differential graded pro-setting, similar to that of \cite[Theorem 23]{P}.

\begin{prop}\label{propLin}
Let $(A,M)$ and $(B,N)$ be SH LR[1] pairs, with $M$ flat as a graded $A$-module, and $N$ flat as a graded $B$-module. Let  $f_0:B\to A$ be a morphism of cdga-s, and $f=(f^n:\op{Sym}_A^n(M)\to A\otimes_B(N))_{n>0}$ a family of graded $A$-module morphisms. Denote by $F:\op{Sym}_A(M)\to \op{Sym}_A(A\otimes_BN)$ the  graded coalgebra map cogenerated by $f$ (Equation \ref{extension}), and by
$$F^*:A\otimes_B\widehat{Sym}_B(N^*)=\widehat{Sym}_A(A\otimes_B N^*)\to\widehat{Sym}_A(M^*)$$
the morphism of pro-finite graded $A$-algebras dual to $F$.
$F^*$, identified via the base change adjunction \cite[Proposition 11]{P} to a morphism of graded pro-$A$-algebras
$\widehat{Sym}_BN^*\to\widehat{Sym}_AM^*$,
respects the Chevalley-Eilenberg differential if and only if $(f_0,f)$ is a linear SH morphism.
\end{prop}

Although pullbacks of general SH LR pairs are not well defined, linear SH morphisms whose target is a dg LR pair factor uniquely through the pullback.
\begin{prop}\label{fact}
Given a SH LR pair $(A,M)$,\footnote{i.e. a SH LR [1] pair $(A,M[1])$} a dg LR pair $(B,N)$, and a linear SH morphism $(f_0,f):(A,M)\rightsquigarrow (B,N)$, there exists a unique linear SH morphism $(\op{Id}_{A},\overline{f}):(A,M)\rightsquigarrow (A,f_0^!N)$ for which $$(f_0,f)=(f_0,p)\circ (\op{Id}_{A},\overline{f}).$$
\end{prop}
\begin{proof}
Morphisms $(\overline{f}^k:\op{Sym}_A^k(M[1])\to f_0^!N[1])_{k\in\mathbb{N}}$ are defined as universal arrows in the category of graded $A$-modules
\begin{center}
\begin{tikzcd} \op{Sym}_A^k(M[1]) \arrow[drr, bend left, "f^k"] \arrow[ddr, bend right, "\Gamma^k"] \arrow[dr, dotted, "{\overline{f}^k}" description] & & \\
 &{f_0}^!N[1] \arrow[r, "p"] \arrow[d] & A\otimes_B N[1]\arrow[d]\\ 
& \op{Der}(A,A)[1] \arrow[r] & \op{Der}(B,A)[1]. \end{tikzcd}
\end{center}
Property (1) of the Definition \ref{defLin} is clearly satisfied. Property (2) is verified by direct computation.
%
\end{proof}

\section{Equivalence of dg and SH LR pairs}\label{three}
Equivalence of dg Lie and $L_{\infty}$-algebras, which is a consequence of the the bar-cobar construction for the operad $Lie$ \cite[Chapter 11]{Valette}, is highly relevant to the present section. Let us begin with a brief recollection.

Let $L$ be an $L_{\infty}$ algebra, and let $C_k(L)$ be its reduced Chevalley Eilenberg coalgebra. Cobar complex of $C_k(L)$  is the free graded Lie algebra on $C_k(L)[-1]$, with the differential generated by
$$d_{\mathcal{L}_k(C_k(L))}(s x)=-s d_{C_k(L)}x-\frac{1}{2}\sum_i(-1)^{|x^{(1)}_i|}[(s x)_i^{(1)},(s x)_i^{(2)}],\hspace{5pt} x\in C_k(L),$$
where $\overline{\Delta}(x)=\sum_i x_i^{(1)}\otimes x_i^{(2)}$ is the reduced coproduct.
Later dg Lie algebra is denoted by $\mathcal{L}_k(C_k(L))$. 
Denote by ${\tt dgLie}(k)$ the category of dg Lie algebras, by ${\tt Lie}_{\infty}(k)$ the category of $L_\infty$-algebras with $L_\infty$-morphisms, and by $\imath$ the inclusion of the former category into the later.
As a consequence of the bar-cobar adjunction for Lie algebras, assignment $L\mapsto\mathcal{L}_k\circ C_k(L)$ is left adjoint to the inclusion
\begin{equation}\label{adjLie}\mathcal{L}_k\circ C_k:{\tt Lie}_{\infty}(k)\leftrightarrows  {\tt dgLie}(k):\imath.\end{equation}
Moreover, adjunction's unit is an infinity quasi-isomorphism of $L_\infty$ algebras $L\rightsquigarrow\mathcal{L}_k(C_k(L))$, and the adjunction's counit is a quasi-isomorphism of dg Lie algebras $\mathcal{L}_k(C_k(L))\to L$. Thus, the above adjunction yields an equivalence under the Dwyer-Kan localization. \bigskip 

Let now $(A,M)$ be a SH LR pair. In particular, $M$ is an $L_{\infty}$-algebra, and the anchor is an $L_{\infty}$-morphism $\Gamma:M\rightsquigarrow\op{Der}(A)$. Denote by $\eta_k$ and $\epsilon_k$ respectively the unit and counit of the adjunction \ref{adjLie}. Observe that $\Gamma$ factors as
\begin{equation}\label{diag1}
\begin{tikzcd}[column sep=large]
M\arrow[d, rightsquigarrow,"\eta_k"]\arrow[drr,rightsquigarrow,"\Gamma"] & &  \\
\mathcal{L}_kC_k(M)\arrow[r,swap,"\mathcal{L}_kC_k(\Gamma)"]& \mathcal{L}_kC_k(\op{Der(A)}) \arrow[r, swap,"\epsilon_k"] &\op{Der}(A)
\end{tikzcd}
\end{equation}
 Denote the horizontal composition by
$$\gamma:\mathcal{L}_kC_k(M)\to\op{Der}(A).$$ As a map of graded Lie algebras, $\gamma$ is uniquely determined by its restriction to the generators, explicitly by the morphism
$$\gamma_{\op{gen}}:C_k(M)[-1]\to\op{Der}(A),\hspace{5pt} s(m_1\odot\cdots m_n)\mapsto s\Gamma^n(m_1,\ldots,m_n)$$
of graded vector spaces.
Denote by $C_A(M)$ the graded-symmetric $A$-coalgebra $\op{Sym}_A(M[1])$, which is itself a quotient of $\op{Sym}_k(M[1])$ -- the underlying graded coalgebra of $C_k(M)$.  As $\Gamma^n$ are $A$-multilinear, the map $\gamma_{\op{gen}}$ factors through $C_A(M)[-1]$. Hence, $C_A(M)[-1]$ is a graded anchored $A$-module. $\mathcal{L}_AC_A(M):=LR(C_A(M)[-1])$ denotes the free LR algebra of  $C_A(M)[-1]$. Clearly, $\mathcal{L}_A(C_A(M))$ is a quotient of the free graded Lie algebra $\mathcal{L}_k(C_k(M))$. Although the differential in the Chevalley-Eilenberg coalgebra $C_k(M)$ doesn't pass to the quotient $C_A(M)$, it happens that the total differential on $\mathcal{L}_k(C_k(M))$ does pass to the quotient $\mathcal{L}_A(C_A(M))$:
\begin{lemma}\label{LtoLR}
Differential on $\mathcal{L}_k(C_k(M))$ passes to the quotient $\mathcal{L}_A(C_A(M))$.
\end{lemma}
\begin{proof}
Recall that, given a graded $A$-module $N$, the $A$-module structure on $N[-1]$ is defined by $as n=(-1)^{|a|}s(an)$. Denote by $p$ the projection $p:\mathcal{L}_k(C_k(M))\to\mathcal{L}_A(C_A(M))$.
Observe that 
\begin{equation*}
\begin{split}
&\frac{1}{2}\sum_{\substack{k=1,\ldots n\\ \sigma\in \op{Sh}(k,n-k)}}(-1)^{|m_1,\ldots m_n|_\sigma+m_{\sigma(1)\ldots\sigma(k)}}[sm_{\sigma(1)\ldots\sigma(k)},sm_{\sigma(k+1)\ldots\sigma(n)}]\\
=&\sum_{\substack{k=1,\ldots n\\ \sigma\in \op{Sh}(k,n-k), \sigma(1)=1}}(-1)^{|m_1,\ldots m_n|_\sigma+m_{\sigma(1)\ldots\sigma(k)}}[sm_{\sigma(1)\ldots\sigma(k)},sm_{\sigma(k+1)\ldots\sigma(n)}],
\end{split}
\end{equation*}
Below, $\tau$ denotes the permutation in hands. For instance, given $\sigma\in S_3$,  in $$(-1)^{|m_{1\ldots 3}|_{\tau}}m_{\sigma(2)}\odot m_{\sigma(1)}\odot am_{\sigma(3)},$$ $\tau:=\sigma\circ(1,2)$.    We get
\begin{align*}
(p\circ d)&s(am_{1}\odot m_2\odot\cdots\odot m_n)\\=
&-\sum_{\substack{k=1\ldots n\\ \sigma\in \op{Sh}(k,n-k),\sigma (1)=1}}(-1)^{|m_1\ldots m_{n}|_\sigma}s(D^k(am_{\sigma(1)\ldots \sigma(k)}) \odot m_{\sigma(k+1)\ldots\sigma(n)})\\
&-\sum_{\substack{k=1\ldots n-1\\ \sigma\in \op{Sh}(k,n-k), \sigma(k+1)=1}}(-1)^{|m_1\ldots m_n|_\sigma+|a|(|m_{\sigma(1)\ldots\sigma(k)}|)}s(D^k(m_{\sigma(1)\ldots\sigma(k)}) \odot am_{\sigma(k+1)\ldots\sigma(n)})\\
&-\sum_{\substack{k=1\ldots n-1\\ \sigma\in \op{Sh}(k,n-k), \sigma(1)=1}}(-1)^{|m_1\ldots m_n|_\sigma+|m_{\sigma(1)\ldots\sigma(k)}|}[as m_{\sigma(1)\ldots\sigma(k)},sm_{\sigma(k+1)\ldots \sigma(n)}]\\
=&-\sum_{\substack{k=1\ldots n-1\\ \sigma\in \op{Sh}(k,n-k),\sigma (1)=1}}(-1)^{|m_1\ldots m_n|_\sigma}as(D^k(m_{\sigma(1)\ldots\sigma(k)}) \odot m_{\sigma(k+1)\ldots\sigma(n)})\\
&
-\sum_{\substack{k=1\ldots n-1\\ \sigma\in \op{Sh}(k,n-k),\sigma (1)=1}}(-1)^{|m_1\ldots m_n|_{\tau}+(|a|+1)(|m_{\sigma(2)\ldots\sigma(k)}|+1)}\rho^{k-1}(m_{\sigma(2)\ldots\sigma(k)})(a) sm_{\sigma(1)}\odot m_{\sigma(k+1)\ldots\sigma(n)}\\
&-\sum_{\substack{k=1\ldots n-1\\ \sigma\in \op{Sh}(k,n-k), \sigma(k+1)=1}}(-1)^{|m_1\ldots m_n|_\sigma}as(D^k(m_{\sigma(1)\ldots\sigma(k)}) \odot m_{\sigma(k+1)\ldots\sigma(n)})\\
&
-\sum_{\substack{k=1\ldots n-1\\ \sigma\in \op{Sh}(k,n-k), \sigma(1)=1}}(-1)^{|m_1\ldots m_n|_\sigma+|m_{\sigma(1)\ldots\sigma(k)}|}a[sm_{\sigma(1)\ldots\sigma(k)},sm_{\sigma(k+1)\ldots\sigma(n)}]\\
&
+\sum_{\substack{k=1\ldots n-1\\ \sigma\in \op{Sh}(k,n-k), \sigma(1)=1}}(-1)^{|m_1\ldots m_n|_\tau+(|a|+1)(|m_{\sigma(k+1)\ldots\sigma(n)}|+1)}\rho^{n-k}(m_{\sigma(k+1)\ldots\sigma(n)})(a)sm_{\sigma(1)\ldots\sigma(k)}\\
=&
(-1)^{|a|}d_Aa\cdot p(sm_{1\ldots n})+a\cdot(p\circ d)(sm_{1\ldots n}).
\end{align*}
This shows that the differential's restriction to $C_k(M)[-1]$ descends to the quotient into an $A$-module derivation. By induction, it is readily verified that the differential preserves the ideals generated by \ref{rel1} and \ref{rel2}. 
\end{proof}

Diagram \ref{diag1} is hence extended to
\begin{center}
\begin{tikzcd}\label{diag2}
M\arrow[d, rightsquigarrow,"\eta_k"]\arrow[drr,rightsquigarrow,"\Gamma_M"] & &  \\
\mathcal{L}_kC_k(M)\arrow[r,swap,"\mathcal{L}_kC_k(\Gamma)"]\arrow[d,"p"]& \mathcal{L}_kC_k(\op{Der(A)}) \arrow[r, swap,"\epsilon_k"] &\op{Der}(A).\\
\mathcal{L}_AC_A(M)\arrow[urr,swap,"\Gamma_{\mathcal{L}_AC_A(M)}"]&  &
\end{tikzcd}
\end{center}
Denote by $\eta_A$ the $L_\infty$ - morphism $p\circ \eta_k: M\to \mathcal{L}_AC_A(M)$, explicitly $$\eta_A^n(m_1\odot\cdots\odot m_n)=m_1\odot\cdots\odot m_n\in C_A(M)\subset \mathcal{L}_AC_A(M).$$
By remark \ref{rem11}, $\eta_A$ is a morphism of SH LR pairs. In fact, if $M$ is a cofibrant $A$-module, it is a SH quasi-isomorphism:

\begin{lemma}\label{Lemma2}
Given a SH LR pair $(A,M)$, with  $M$ a cofibrant $A$-module, the map
$\eta_A:M\rightsquigarrow \mathcal{L}_AC_A(M)$ is a SH quasi-isomorphism.
\end{lemma}

\begin{proof}
Equip $\mathcal{L}_AC_A(M)$ with the total filtration, so that, for example $[m_1,m_2\odot m_3]\in F^3\mathcal{L}_AC_A(M)$. Formally, the filtration is defined as follows.  As a graded $A$-module, 
$$\mathcal{L}_AC_A(M)=LR(\op{Sym}_A(M[1])[-1]).$$
Starting from the filtration $\ref{filt1}$ on a free LR pair, define $F^{m,n}$ recursively on $m$ by
\begin{itemize}
\item
$F^{1,n}LR(\op{Sym}_A(M[1])[-1])=\op{Sym}^{\leq n}_A(M[1])[-1]$; and
\item
for $m>1$, $F^{m,n}LR(\op{Sym}_A(M[1])[-1])$ is the submodule of $LR^{\leq m}(\op{Sym}_A(M[1])[-1])$ generated by $F^{1,n}LR(\op{Sym}_A(M[1])[-1])$, and $$\{[x,x']:\hspace{10pt} x\in F^{i,k} LR(\op{Sym}_A(M[1])[-1]),\hspace{5pt} x'\in F^{j,l}LR(\op{Sym}_A(M[1])[-1]),\hspace{5pt} i+j=m,\hspace{5pt} k+l=n\}.$$
\end{itemize}
Finally, the total filtration is defined as
$$ F^{n}LR(\op{Sym}_A(M[1])[-1])=\bigcup_{m\in \mathbb{N}}F^{ m, n}LR(\op{Sym}_A(M[1])[-1]).$$

Clearly, the differential respects the total filtration, giving rise to the filtered complex $(F^{ n}\mathcal{L}_AC_A(M))_n$. If the anchor in $M$ is trivial, so is the anchor in $\op{Sym}_A(M[1])[-1]$, and consequently, the free graded LR pair over it is the free graded Lie $A$-algebra, hence equipped with a weight grading in place of a filtration. In this situation, the total filtration in the above construction can also be replaced by a total weight grading on $LR(\op{Sym}_A(M[1])[-1])$. If $n$-ary brackets for $n>1$ are trivial as well, weight grading is respected by the differential.

For arbitrary $M$, denote by $M_{\op{tr}}$ the dg LR pair with the same underlying dg $A$-module $M$, but with trivial bracket and anchor. The associated graded complex of the filtration $(F^{ n}\mathcal{L}_AC_A(M))_n$ is the graded complex $\mathcal{L}_AC_A(M_{\op{tr}})$.

Finally, the unit's $\eta_A:M\rightsquigarrow \mathcal{L}_AC_A(M)$ first component $\eta_A^1$ is the inclusion $$M=F^{1}\mathcal{L}_AC_A(M)\hookrightarrow \mathcal{L}_AC_A(M).$$ By a spectral sequence argument, it is a quasi-isomorphism if the associated graded complex of $(F^{ n}\mathcal{L}_AC_A(M))$ is acyclic in weights other then one, equivalently if $\mathcal{L}_AC_A(M_{\op{tr}})$ is acyclic in those weights. 

As cofibrant modules are retracts of cell complexes, and  retracts of acyclic modules are acyclic, it is safe to assume that $M$ is a cell complex in $\op{Mod}(A)$ in the sense of \cite{DPP}, i.e. graded-free $A$-module with a well ordered set of generators for which the differential satisfies a lowering condition
$$M=A\langle m_i \rangle_{i\in I},\hspace{10pt} d(m_i)\in A\langle m_i \rangle_{j<i}.$$

Denote by $(\mathcal{L}_AC_A)^n(M_{\op{tr}})$ its weight $n$ component for $n\neq 1$. \emph{In the special case} $d_M(m_i)=0$, or more generally, if $M_{\op{tr}}=A\otimes_k V$ for a dg vector space $V$ -- viewed as a dg Lie algebra with trivial bracket,  \begin{equation}\label{equiv}(\mathcal{L}_AC_A)^n(M_{\op{tr}})=A\otimes_k(\mathcal{L}_kC_k)^n(V).\end{equation}
The right-hand-side is acyclic by an operadic argument, see for example \cite[Theorem 11.4.7]{Valette}. General case follows from here by the transfinite induction. Namely, denote $V=k\langle m_i\rangle_{i\in I}$. On the level of graded Lie algebras, equation \ref{equiv} is still satisfied, and yields
  \begin{equation}\label{equiv2}(\op{Lie}_A\op{Sym}_A)^n(M[1])[-1]=A\otimes_k(\op{Lie}_k\op{Sym}_k)^n(V[1])[-1].\end{equation}
Basis of $\op{Lie}_k\op{Sym}_k(V[1])[-1]$ is given by Hall trees -- certain binary trees whose vertices are labeled by the monomials in $\op{Sym}_k(M[1])$ with trivial coefficients. Given such tree, define its \emph{label} as the multiset of all the variables which appear in the labels of the vertices. For example, label of $[[m_1,m_2],m_2\odot m_3]$ is $\{m_3,{m_2}^2,m_1\}$.  As in the example, labels will be ordered in the decreasing order. $(\op{Lie}_k\op{Sym}_k)^n(V[1])[-1]$ is spanned by Hall trees whose label is of cardinality $n$.  Ordered labels of the Hall trees in $(\op{Lie}_k\op{Sym}_k)^n(V[1])[-1]$ form a well ordered set for the lexicographic ordering, denoted by $\Lambda$. Given a label $a\in \Lambda$, we write
$x\in (\op{Lie}_k\op{Sym}_k)^n_{\leq a}(V[1])[-1]$ if $x$ is a linear combination of Hall trees whose labels are smaller or equal to $a$. Due to cell complex differential's decreasing property, this induces a filtration $((\mathcal{L}_AC_A)^n_{\leq a}(M_{\op{tr}}))_{a\in \Lambda}$ on $(\mathcal{L}_AC_A)^n(M_{\op{tr}})$. Acyclicity is proven by induction on the filtration. Following observation is relevant for the proof:

Denote by $V_0$ the graded vector space $V$ with the zero differential and by $M_0$  the tensor product $A\otimes_k V_0$, with trivial bracket and anchor. As differential in $(\mathcal{L}_AC_A)^n(M_0)$ does not change the label, $(\mathcal{L}_AC_A)^n(M_0)$ is a direct sum of dg submodules spanned by Hall trees with a fixed label  $$(\mathcal{L}_AC_A)^n(M_0)=\oplus_{a\in \Lambda} (\mathcal{L}_AC_A)^n_{ a}(M_0),$$
in which each component is acyclic as a direct summand of an acyclic chain complex (see the above \emph{special case}).

Coming back to the induction, for $0\in \Lambda$ the minimal element,  $$(\mathcal{L}_AC_A)^n_{\leq 0}(M_{\op{tr}})=(\mathcal{L}_AC_A)^n_{ 0}(M_0)$$ is acyclic. Assume $(\mathcal{L}_AC_A)^n_{<a }(M_{\op{tr}})$ to be acyclic for some $a\in A$, and consider the two-step filtration
$$(\mathcal{L}_AC_A)^n_{<a}(M_{\op{tr}})\hookrightarrow (\mathcal{L}_AC_A)^n_{\leq a}(M_{\op{tr}}).$$
As  $$(\mathcal{L}_AC_A)^n_{\leq a}(M_{\op{tr}})/ (\mathcal{L}_AC_A)^n_{<a}(M_{\op{tr}})=(\mathcal{L}_AC_A)^n_{a}(M_0),$$ $(\mathcal{L}_AC_A)^n_{\leq a}(M_{\op{tr}})$ is acyclic as well.
\end{proof}

Finally the ground is set to prove the  equivalence between dg LR pairs and SH LR pairs with linear morphisms.
\begin{theo}\label{theo1}
\begin{enumerate}
\item
 The inclusion $\imath:{\tt dgLR}(k)\to {\tt SHLR}^{\op{lin}}(k)$ from the category of dg LR pairs to the category of SH LR pairs with linear SH morphisms admits a left adjoint, denoted by $\mathcal{L}_{-}C_{-}$ which to a SH LR pair $(A,M)$ associates the dg LR pair $(A,\mathcal{L}_{A}C_{A}(M))$ defined above.
\item
Denote by ${\tt dgLR}^{\op{cof}}(k)$ the full subcategory of ${\tt dgLR}(k)$ whose objects are pairs $(A,M)$, with $A$ a semi-free cdga, and $M$ a cell complex in $A$-modules. The full subcategory ${\tt SHLR}^{\op{lin,cof}}(k) \subset{\tt SHLR}^{\op{lin}}(k)$ is defined like-wise. The adjunction $\mathcal{L}_{-}C_{-}\dashv \imath$ restricts to
\begin{equation}\label{main_adj}\mathcal{L}_{-}C_{-}:{\tt SHLR}^{\op{lin,cof}}(k)\rightleftarrows  {\tt dgLR}^{\op{cof}}(k):\imath,\end{equation}
inducing an equivalence between their Dwyer-Kan localizations.
\end{enumerate}
\end{theo}
\begin{proof}
\begin{enumerate}
\item Using the definition of adjunction via universal morphism, it remains to show that, given a linear SH morphism $(f_0,f):(A,M)\rightsquigarrow (B,N)$, with $(B,N)$ a dg LR pair, there exists a unique morphism of dg LR pairs $(f_0,\overline{f}):(A,\mathcal{L}_{M}C_{A}(M))\to(B,N)$ such that the following diagram commutes:
\begin{center}
\begin{tikzcd}
(A,M)\arrow[d, rightsquigarrow,"(\op{id}{,}\eta_A)"]\arrow[drr,rightsquigarrow,"(f_0{,}f)"] &  &\\
(A{,}\mathcal{L}_AC_A(M))\arrow[rr,swap,"(f_0{,}\overline{f})"]& &(B,N).
\end{tikzcd}
\end{center}
Due to Proposition \ref{fact}, the proof reduces to the case $A=B$, $f_0=\op{id}$.
In this situation, $f$ is in particular an infinity morphism of $L_\infty$-algebras, hence factors as 
$$M\overset{\eta_k}{\rightsquigarrow}\mathcal{L}_kC_k(M)\xrightarrow{\tilde{f}}N,$$
where $\tilde{f}$ is uniquely determined as a morphism of graded Lie algebras on the generators in $\op{Sym}_k(M[1])[-1]$, where it reads
$$\tilde{f}(s(m_1\odot\cdots\odot m_n))=s f^n(m_1\odot\cdots\odot m_n).$$
Since $f^n$ are $A$-multilinear, the restriction of $\tilde{f}$ to $\op{Sym}_k(M[1])[-1]$ factors through $\op{Sym}_A(M[1])[-1]$ via $A$-linear map, denoted by $\overline{f}_{\op{gen}}$. As $N$ is a dg LR pair, its higher anchors vanish, hence 
$$\Gamma \circ f^n(m_1\odot\cdots\odot m_n)=\Gamma^n(m_1\odot\cdots\odot m_n).$$
Thus, $\overline{f}_{\op{gen}}$ is a map of anchored graded $A$-modules, and uniquely determines a morphism of graded LR pairs $\overline{f}:LR\op{Sym}_A(M[1])[-1]\to N$. It is left to verify that $\overline{f}$ respects the differential. However, this is immediate, since $\tilde{f}$, which is the composition
$$\mathcal{L}_kC_k(M)\overset{p}{\twoheadrightarrow} \mathcal{L}_AC_A(M)\xrightarrow{\overline{f}} N$$ does.\medskip

\item Recall from \cite{P} that SH quasi-isomorphisms in enjoy the 2-out-of-3 property. By Lemma \ref{Lemma2}, adjunction's unit is a linear SH quasi-isomorphism. From the naturality of the unit and 2-out-of-3 property, it follows that $\mathcal{L}_{-}C_{-}$ sends linear SH quasi-isomorphisms to quasi-isomorphisms. For a dg LR pair $(A,M)$ denote by $\epsilon_A:\mathcal{L}_{A}C_{A}(M)\to M$ the adjunction's unit. As with any adjunction, the composition
$$M\overset{\eta_A}{\rightsquigarrow} \mathcal{L}_{A}C_{A}(M)\xrightarrow{\epsilon_A}M$$
is identity. By the 2-out-of-3 property, unit is a quasi-isomorphism of dg LR pairs. Finally, \cite[Corollary 3.6]{DK} applies, and the adjunction \ref{main_adj} induces and equivalence between the Dwyer-Kan localizations of ${\tt SHLR}^{\op{lin,cof}}(k)$ and ${\tt dgLR}^{\op{cof}}(k)$.
\end{enumerate}
\end{proof}
\begin{rem}\label{remark}
\begin{enumerate}
\item
For a SH LR pair $(A,M)$ with $M$ a cell complex of $A$-modules, $\mathcal{L}_AC_A(M)$ is not only a  cell complex of $A$-modules, but also a cell complex in the Nuiten's semi-model structure on dg Lie-Rinehart algebras over $A$. Hence, the equivalence is threefold, as redefining ${\tt dgLR}^{\op{cof}}(k)$  as consisting of pairs $(A,M)$ with $A$ semi free, and $M$ a cell complex in Nuiten's semi-model structure yields the same result. Letter interpretation is used in section \ref{cart}.
\item That $A$ is semi-free is not used. Similarly, the construction applies if $M$ is only a cofibrant $A$-module instead of a cell complex. Hence, the above proof gives three threefold equivalences for different cofibrancy conditions. However, it is not clear if the resulting nine infinity categories are all equivalent. Reason that we emphasize one particular cofibrancy condition is that it is the one for which equivalence extends to general morphisms.
\end{enumerate}
\end{rem}
Finally, we turn our attention to general morphisms.
Let $A=(k[x_i]_{i\in I},d_A)$ be a semi-free cdga, and $M=(A\langle m_j\rangle_{j\in J},d_M)$ a cell complex in the category of $A$-modules. Endow $(A,M)$ with the structure of a SH LR pair, and denote by $d_{CE}$ the corresponding Chevalley-Eilenberg differential on $\widehat{\op{Sym}}_AM^*$. Denote by $\Omega_A$ the graded module of K\"ahler differentials on $A$. Denote by $\xi_i$ the generator $dx_i\in \Omega_A$, and by $c_i$ its $(1)$-shift.  Define another SH LR pair $(A,\overline{M})$ in terms of its Chevalley-Eilenberg complex as $\widehat{\op{Sym}}_A(M^*\oplus \Omega_A^*\oplus\Omega_A^*[-1])$, with the differential
\begin{equation}\label{ass}
\begin{split}
d(x_i)&=c_i^*+\sum_{\substack{n\in\mathbb{N}\\i_1<\ldots< i_n\in I}}\frac{1}{n_1!\cdots n_n!}(\xi_{i_1}^*\partial_{i_1})^{n_1}\cdots(\xi_{i_n}^*\partial_{i_n})^{n_n}d_{CE}(x_i);\\
d(m_j^*)&=\sum_{\substack{n\in\mathbb{N}\\i_1<\ldots< i_n\in I}}\frac{1}{n_1!\cdots n_n!}(\xi_{i_1}^*\partial_{i_1})^{n_1}\cdots(\xi_{i_n}^*\partial_{i_n})^{n_n}d_{CE}(m_j^*);\\
d(\xi_i^*)&=-c_i^*;\hspace{15pt}d(c_i^*)=0.\\
\end{split}
\end{equation}
The differential operator
$$\sum_{\substack{n\in\mathbb{N}\\i_1<\ldots< i_n\in I}}\frac{1}{n_1!\cdots n_n!}(\xi_{i_1}^*\partial_{i_1})^{n_1}\cdots(\xi_{i_n^*}\partial_{i_n})^{n_n}$$ renames each subset of variables $x_k$ in a monomial to $\xi_k^*$, then sums up. For example,
$$x_1^2x_2 m_j\mapsto x_1^2x_2 m_j+ 2x_1\xi_1^* x_2 m_j +  2x_1\xi_1^* \xi_2^* m_j + {\xi_1^*}^2 x_2 m_j + x_1^{2}\xi_2^* m_j +{\xi_1^*}^2 \xi_2^* m_j.$$
 
Denote by ${\tt SH LR}^{\op{cof}}$ the category of SH LR pairs $(A,M)$ with $A$ a semi-free cdga, and $M$ a cell complex in the category of $A$-modules. We first show that the assignment $(A,M)\mapsto (A,\overline{M})$ constitutes the right adjoint to the inclusion of a wide sub-category
\begin{equation}\label{adjunkcija2}\imath:{\tt SH LR}^{\op{lin,cof}}(k)\rightleftarrows {\tt SH LR}^{\op{cof}}(k):(-,\overline{\hspace{2pt}\cdot\hspace{2pt}}).\end{equation}
For this, it suffices to construct the adjunction's counit $\epsilon_{(A,M)}:(A,\overline{M})\to (A,M)$ which satisfies the following universal property: Given a SH LR pair $(B,N)$, and a general SH morphism $f:(B,N)\rightsquigarrow (A,M)$ there exists a unique linear SH morphism $\overline{f}:(B,N)\to (A,\overline{M})$ such that $f=\epsilon\circ\overline{f}.$

Counit $\epsilon$ is defined as the map $$\epsilon^*:\widehat{\op{Sym}}_A M^*\to \widehat{\op{Sym}}_A (M^*\oplus \Omega_A^*\oplus \Omega_A^*[-1]),\hspace{10pt}\epsilon^*(x_i)=x_i+\xi_i^*,\hspace{10 pt} \epsilon^*(m_i^*)=m_i^*.$$
With this, the unique map $\overline{f}$ is defined as\footnote{${f^*_k}$ denotes the weight $k$ component of ${f^*}$}
$$\overline{f}^*: \widehat{\op{Sym}}_A (M^*\oplus \Omega_A^*\oplus \Omega_A^*[-1])\to \widehat{\op{Sym}}_B (N^*), $$ $$\overline{f}^*(x_i)={f^*_0}(x_i),\hspace{10pt}\overline{f}^*(\xi_i^*)={f^*_{>0}}(x_i),\hspace{10pt}\overline{f}^*(c_i^*)=d_{CE}({f^*_{>0}}(x_i))\hspace{10pt}\overline{f}^*(m_j^*)={f^*}(m_j^*),$$  Clearly, adjunction's counit is a weak equivalence. By the 2-out-of-3 property, the same holds for the adjunction's unit, hence the adjunction is an equivalence on the level of infinity categories. This settles the proof of the following theorem:

\begin{theo}\label{theo2}
The inclusion of a wide subcategory  $${\tt SH LR}^{\op{cof,lin}}(k)\hookrightarrow  {\tt SH LR}^{\op{cof}}(k)$$
is under Dwyer-Kan localization an equivalence of infinity categories.
\end{theo}
\section{Cofibrations, factorizations, and lifting properties}\label{lift}

First, we deal with morphisms of SH LR pairs over a fixed base cdga. In this context, a SH LR[1] pair $(A,M)$ is called a SH LR[1] algebra $M$ over $A$. A morphism $(\op{id},f):(A,M)\to (A,N)$ is denoted by $f$.

\begin{defi}
Given SH LR[1] algebras $M$ and $N$ over $A$, a strict morphism $f:M\to N$ of SH LR[1] algebras ($f^1=f$, $f^i=0$ for $i\neq 1$) is called
\begin{enumerate}
\item a \emph{cofibration} if it is a relative cell complex of $A$-modules; and
\item a \emph{fibration} if it is a fibration of $A$-modules.
\end{enumerate}
In particular, $M$ is cofibrant if it is a cell complex of $A$-modules.
\end{defi}
\begin{rem} All the results of the present section remain valid when cofibrations are defined as cofibrations of $A$-modules in place of relative cell complexes. However, proofs become more involved.
\end{rem}
The following lemma shows that morphisms from a cofibrant SH LR[1] $A$-algebra decompose into a cofibration followed by a weak equivalence:

\begin{lemma}\label{lemma18}
Given $A\in{\tt cdga}(k)$, let $M$ and $N$ be SH LR[1] algebras over $A$, with $M$ cofibrant, and let $f=(f^n)_{n\in \mathbb{N}}:M\rightsquigarrow N$ be a linear SH[1] morphism. Given a decomposition in ${\tt Mod}(A)$ of $f^1$ 
$$M\overset{\imath}{\hookrightarrow}\overline{M}\overset{p^1}{\twoheadrightarrow} N$$
into a relative cell complex followed by a trivial fibration, there exists a SH LR[1] algebra structure on $\overline{M}$ such that $\imath$ is a (strict) SH LR[1] morphism, and an extension $p=(p^n)_{n\in \mathbb{N}}$ of $p^1$ into a linear SH[1] morphism, such that $f=p\circ\imath$. If $f^k=0$ for $k>k_0$, for appropriate SH LR[1] structure on $\overline{M}$, $p$ can be chosen to satisfy the same property.

\end{lemma}
\begin{proof}
 As a cell complex, $M=A\langle v_j \rangle_{j\in J}$, with a differential which satisfies the lowering condition $d(v_j)\in A\langle v_k\rangle_{k<j}.$ Denote $\overline{M}=(M\oplus A\langle x_i \rangle_{i\in I},d)=(A\langle x_i,v_j \rangle_{i,j},d)$, with $I$ a well-ordered set with respect to which $d$ satisfies the lowering condition $d(x_i)\in M\oplus A\langle x_l \rangle_{l<i}$. 
 By definition, for $n>1$, $p^n$ is the extension of $f^n$ by zero to the generators $x_i$, and the anchor is $\Gamma=\Gamma_N\circ p$. Clearly, $f^k=0$ implies $p^k=0$.

In weight one, $D^1$ is the differential on $\overline{M}$. Next, we define $D^2$. For all $i_1,i_2\in I$ (resp. $i\in I, j\in J$), fix any $\overline{m}_{i_1i_2}$ (resp. $\overline{m}_{ij}$)$\in \overline{M}$ such that 
\begin{equation*}
\begin{split}
&p^1(\overline{m}_{i_1i_2})=D^1(p^2(x_{i_1}\odot x_{i_2}))+D^2(p^1(x_{i_1})\odot p^1(x_{i_2}))-p^2(D^1(x_{i_1})\odot x_{i_2})-(-1)^{|x_{i_1}|}p^2(x_{i_1}\odot D^1(x_{i_2})),\\
&p^1(\overline{m}_{ij})=D^1(p^2(v_{i}\odot x_{j}))+D^2(p^1(v_{i})\odot p^1(x_{j}))-p^2(D^1(x_{i})\odot v_{j})-(-1)^{x_i}p^2( x_{i}\odot D^1(v_{j})).
\end{split}
\end{equation*} Define provisional differential $\tilde{D}^2$ as the $k$-linear graded symmetric map such that, with $\rho^1$ defined as in Equation \ref{SHLRprop},
\begin{equation*}
\begin{split}
&\tilde{D}^2(ax_i,b v_j)=(-1)^{|a|+(1+|x_i|)|b|}ab\overline{m}_{ij}+ (-1)^{(|b|+|v_j|)(|a|+|x_i|)+|b|}b\rho^1(v_j)(a)\cdot x_i+(-1)^{|a|}a\rho^1(x_i)(b)\cdot v_j;\\
&\tilde{D}^2(ax_{i},b x_{j})=(-1)^{|a|+(1+|x_{i}|)|b|}ab\overline{m}_{ij}+ (-1)^{(|b|+|x_{j}|)(|a|+|x_{i}|)+|b|}b\rho^1(x_{j})(a)\cdot x_{i}+(-1)^{|a|}a\rho^1(x_{i})(b)\cdot x_{j};\\
&\tilde{D}^2(av_{i},b v_{j})=D^2((av_{i})\odot (b v_{j})).
\end{split}
\end{equation*}
By construction, $[\tilde{D}^2,D^1]$ is a morphism of dg $A$-modules, hence a cycle in $\op{Hom}_A(\overline{M}^{\otimes 2},\overline{M})$), and $p^1\circ [\tilde{D}^2,D^1]=0$. Further on, $ [\tilde{D}^2,D^1]\circ \imath=\imath\circ[D^2_M,D^1_M]=0$. Finally, in the quasi-isomorphism of short exact sequences (homotopy fiber sequences) in $\op{Mod}(A)$
\begin{center}
\begin{tikzcd}
\op{Ker}_{u}\arrow[r,hook]\arrow[d]&\op{Hom}(\overline{M}^{\odot 2},\overline{M})\arrow[r,two heads]\arrow[d,"p^1\circ-"]&\op{Hom} ({M}^{\odot 2},\overline{M})\arrow[d,"p^1\circ-"]\\
\op{Ker}_{d}\arrow[r,hook]&\op{Hom} (\overline{M}^{\odot 2},N)\arrow[r,two heads]&\op{Hom} ({M}^{\odot 2},N),
\end{tikzcd}
\end{center}
$[\tilde{D}^2,D^1]$ is a cycle in $\op{ker}_u$ which maps to zero in $\op{ker}_d$. As the induced map between the homotopy fibers is a quasi-isomorphism, it is a boundary. Hence, there exists an $A$-bilinear map $\delta^2$ such that for $D ^2=\tilde{D}^2+\delta^2$, $[{D}^2,D^1]=0$, $p^1\circ D^2=p^1\circ \tilde{D}^2=p^2\circ D^1+D^2\circ p^1+D^1\circ p^2$, and similarly for $D^2\circ \imath$. 

Assume that $D^k$ are defined up to $k=n$. $D^{n+1}$ is defined as follows. 
For a family of generators in $A\langle x_i, v_j \rangle_{i\in I,j\in J}$,  $(x_1,\ldots, x_k,v_{k+1},\ldots v_{n+1})$, (with at least one generator $x_i$) denote $y_i=x_i$ for $i\leq k$, and $y_j=v_j$ for $j>k$, and fix any $\overline{m}_{y_1\ldots y_{n+1}}$ such that
\begin{equation*}
\begin{split}
p^1(&\overline{m}_{y_1\ldots y_{n+1}})\\
=&\sum_{\substack {l=1,\ldots, n+1\\k_1+\ldots+k_l=n+1}}\sum_{\sigma\in\op{Sh}(k_1,\ldots,k_l)}(-1)^{|y_1\ldots y_{n+1}|_{\sigma}}D^l(p^{k_1}(y_{\sigma(1)}\odot\cdots\odot y_{\sigma(k)})\odot\cdots\odot p^{k_l}(y_{\sigma(n+2-k_l)}\odot\cdots\odot y_{\sigma(k_k)}))\\
-&\sum_{k=1,\ldots,n+1}\sum_{\sigma\in\op{Sh}(k,n+1-k)}(-1)^{|y_1\ldots y_{n+1}|_{\sigma}}p^{n+2-k}(D^{k}(y_{\sigma(1)}\odot\cdots\odot y_{\sigma(k)})\odot y_{\sigma(k+1)}\odot\cdots\odot y_{\sigma(n+1)}).
\end{split}
\end{equation*}
Define a provisional differential $\tilde{D}^{n+1}$ as the $k$-linear graded symmetric map such that, for $y_1,\ldots,y_{n+1}$ as above,
\begin{equation*}
\begin{split}
\tilde{D}^{n+1}&(a_1y_1,\ldots,a_{n+1}y_{n+1})=(-1)^{\sum_{i=1}^{n+1}|a_i|(1+|y_1|+\ldots+|y_{i-1}|)}a_1\cdots a_{n+1}\overline{m}_{y_1\dots y_{n+1}}\\
+&\sum_{k=1}^{n+1}(-1)^{{\sum_{i=1}^{n+1}|a_i|(1+|y_1|+\ldots+|y_{i-1}|)}+|y_k|(|y_{k+1}|+\ldots+|y_{n+1}|)+|a_k|(1+|a_{k+1}|+\ldots+|a_{n+1}|+|y_1|+\ldots+|y_{n+1}|)}\\
&a_1\cdots\hat{a}_k\cdots a_{n+1} \rho^{n+1}(y_1\odot\cdots\odot \hat{y}_k\odot\cdots\odot y_{n+1})(a_k)\cdot y_k;\\
\tilde{D}^{n+1}&(a_1v_{1},\ldots,a_{n+1}v_{n+1})=D^{n+1}(av_1\odot \cdots\odot a_{n+1} v_{n+1}).
\end{split}
\end{equation*}
By construction,
$$\eta=[\tilde{D}^{n+1},D^1]+\frac{1}{2}\sum_{k=2}^n [{D}^{k},D^{n+2-k}]$$
is a $A$-linear, and $ \eta\circ \imath=0$. Somewhat less obvious is that $p^1\circ \eta=0$. Denote $D=D^1+\ldots+D^n+\tilde{D}^{n+1}$.  By induction, $[D,D]$ vanishes in weights $\leq n$. Hence, $p^1\circ\eta$ is the weight $n+1$ component of $\frac{1}{2}p\circ [D_N,D_N]$.  By the construction, $[p,D]$ vanishes in weights $\leq n+1$, hence $p^1\circ\eta$ is further equal to the weight $n+1$ component of $[D,D]\circ p$, which is zero, $N$ being an $L_\infty$ algebra.
 By the Jacobi identity for Nijenhuis–Richardson bracket, $\eta$ is a cycle in $\op{Hom}_A(\overline{M}^{\odot n+1},\overline{M})$. Finally, in the quasi-isomorphism of short exact sequences (homotopy fiber sequences) in $\op{Mod}(A)$
\begin{center}
\begin{tikzcd}
\op{Ker}_{u}\arrow[r,hook]\arrow[d]&\op{Hom}(\overline{M}^{\odot n+1},\overline{M})\arrow[r,two heads]\arrow[d,"p^1\circ-"]&\op{Hom} ({M}^{\odot n+1},\overline{M})\arrow[d,"p^1\circ-"]\\
\op{Ker}_{d}\arrow[r,hook]&\op{Hom} (\overline{M}^{\odot n+1},N)\arrow[r,two heads]&\op{Hom} ({M}^{\odot n+1},N),
\end{tikzcd}
\end{center}
$\eta$ is a cycle in $\op{ker}_u$ which maps to zero in $\op{ker}_d$. As the induced map between the homotopy fibers is a quasi-isomorphism, it is a boundary. Take $\delta \in \op{Ker}_{u}$ with  $[\delta,D^1]=\eta$. $p^1\circ \delta$ is a cycle in $\op{Ker}_{d}$, so there exists a cycle $\delta' \in \op{Ker}_u$ with $p^1\circ \delta=p^1\circ \delta'$. $\Delta:=\delta-\delta'$ satisfies $[\Delta,D^1]=\eta$ and  $p^1\circ \Delta=0$. For $D ^{n+1}=\tilde{D}^{n+1}+\Delta$, 
$$[{D}^{n+1},D^1]+\frac{1}{2}\sum_{k=2}^n [{D}^{k},D^{n+2-k}]=0,$$ 
$$p^1\circ D^{n+1}=p^1\circ \tilde{D}^{n+1}=\sum_{k=2}^n p^k\circ D^{n+2-k}+\sum_{k=1}^{n+1}D^{n+2-k}\circ p^k,$$ and $$D^{n+1}\circ \imath-\imath\circ D^{n+1}=\tilde{D}^{n+1}\circ \imath-\imath\circ {D}^{n+1}=0.$$ 
\end{proof}

The following lemma establishes the lifting properties over a fixed base:
\begin{lemma}\label{lemma19}
Given $A\in {\tt dgca}(k)$, let $K, L, M, N$ be SH LR[1] algebras over $A$, with $M$ cofibrant. Let $f:M\hookrightarrow N$ be a cofibration, and $\phi:L\twoheadrightarrow K$ a fibration. Let  $g:M\rightsquigarrow L$ and $\psi:N\rightsquigarrow K$ be linear SH[1] morphisms with $\phi\circ g=\psi\circ f$. If either $f$ or $\phi$ is a weak equivalence, there exists a linear SH[1] morphism $l:N\to L$ such that $g=l\circ f$, and $\psi=\phi\circ l$.
\begin{center}
\begin{tikzcd}
M\arrow[d, hook,"f"]\arrow[d, hook,swap,"\sim"]\arrow[rr,rightsquigarrow,"g"]  & &L\arrow[d, two heads,"\phi"]&M\arrow[d, hook,"f"]\arrow[rr,rightsquigarrow,"g"]  & &L\arrow[d, two heads,"\phi"]\arrow[d, ,swap,"\sim"]\\
N\arrow[urr,dashed,rightsquigarrow,"l"]\arrow[rr,rightsquigarrow,"\psi"] & &K;&N\arrow[urr,dashed,rightsquigarrow,"l"]\arrow[rr,rightsquigarrow,"\psi"] & &K.
\end{tikzcd}
\end{center}
\end{lemma}
\begin{proof}
Observe that  it suffices to find a morphism $l=(l^n)_n$ of $L_{\infty}$-algebras with each $l^n$  $A$-multilinear, such that $g=l\circ f$ and $\psi=\phi\circ l$. Indeed, the equality $\psi=\phi\circ l$ implies compatibility with the anchor. We proceed by induction. $l^1:N\to L$ is given by the lifting property for $A$-modules. Assuming $l^k$ is defined for $k<n$, all that remains is to find $l^n:\op{Sym}_A^nN\to L$ such that
\begin{equation}\label{e1}[l^n,D^1]=-\frac{1}{2}\sum_{k=2}^{n-1}[l^k,D^{n-k+1}],\end{equation}
\begin{equation}\label{e2}\phi\circ l^n=\psi^n,\end{equation}
and
\begin{equation}\label{e3}l^n\circ f^{\odot n}=g^n.\end{equation}
Observe first that the right-hand-side of equation (\ref{e1}) is a cycle in $\op{Hom}_A(\op{Sym}_A^nN, L)$. Indeed, the commutator $[l^k,d^{n-k+1}]$ is $A$-linear by direct verification, and
\begin{equation*}
\begin{split}
\sum_{k=2}^{n-1}[[l^k,D^{n-k+1}],D^1]&=-\sum_{k=2}^{n-1}[[l^k,D^1],D^{n-k+1}]-\sum_{k=2}^{n-1}[D^{n-k+1},D^1],l^k]\\
&=\sum_{k=2}^{n-1}\sum_{l=2}^{k-1}[[l^l,D^{k-l+1}],D^{n-k+1}]+\frac{1}{2}\sum_{k=2}^{n-1}\sum_{l=2}^{n-k}[[D^{l},D^{n-k-l}],l^k]\\
&=-\frac{1}{2}\sum_{k=2}^{n-1}\sum_{l=2}^{n-k}[D^{l},D^{n-k-l}],l^k]+\frac{1}{2}\sum_{k=2}^{n-1}\sum_{l=2}^{n-k}[D^{l},D^{n-k-l}],l^k]=0.
\end{split}
\end{equation*}
Maps of $A$-modules
 $${f}^{{\odot}n}\circ-:\op{Hom}_A(N^{\odot n},K)\to \op{Hom}_A(M^{\odot n},K);\hspace{10pt}{f}^{{\odot}n}\circ-:\op{Hom}_A(N^{\odot n},L)\to \op{Hom}_A(M^{\odot n},L),$$
are trivial fibrations in the first case and fibrations in the second. Similarly, 
$$-\circ\phi:\op{Hom}_A(N^{\odot n},L)\to \op{Hom}_A(N^{\odot n},K);\hspace{10pt}-\circ\phi:\op{Hom}_A(N^{\odot n},L)\to\op{Hom}_A(N^{\odot n},K)$$
are fibrations in the first case and trivial fibrations in the second. Denote by
$A\langle \psi^n, [\psi^n,D^1]\rangle$
the $0$-disk in the category of graded $A$-modules: the free graded
$A$-module generated by an element $\psi^n$ in degree $0$ and an element
$d\psi^n$ in degree $1$, with differential $d(\psi^n)=[\psi^n,D^1]$.
There is a canonical  morphism of complexes
$$
A\langle \psi^n, [\psi^n,D^1]\rangle \longrightarrow \operatorname{Hom}_A(N^{\odot n}, K),
$$
determined by sending the generators $\psi^n$ and $[\psi^n,D^1]$ to the corresponding
elements of $\op{Hom}_A(N^{\odot n}, K)$. The universal map between (homotopy) pullbacks
\begin{center}
\begin{tikzcd}
\op{Hom}_{A/\psi}(N^{\odot n},L)\arrow[r]\arrow[d, two heads]&\op{Hom}_A(N^{\odot n},L)\arrow[d, two heads,swap, "-\circ\phi"] & \op{Hom}_{A/\psi \circ f}(M^{\odot n},L)\arrow[r]\arrow[d, two heads]&\op{Hom}_A(M^{\odot n},L)\arrow[d, two heads,swap, "-\circ\phi"]  \\
A\langle {\psi}^n,[\psi^n,D^1]\rangle   \arrow[r]&\op{Hom}_A(N^{\odot n},K) & A\langle {\psi}^n\circ f,[\psi^n\circ f,D^1]\rangle   \arrow[r]&\op{Hom}_A(M^{\odot n},K),
\end{tikzcd}
\end{center}
denoted by $u$, is a quasi-isomorphism: if $f$ is a weak equivalence, it is the map between quasi-isomorphic homotopy pullback squares (model squares in \cite{GPP}); and if $\phi$ is a weak equivalence, it is a map between acyclic $A$-modules. We wish to prove that it is also a fibration (degree-wise surjective).
As a graded module, denote $M=A\otimes_k V_M$ for a graded vector space $V_M$. As a relative cell complex, $f$ is the map $A\otimes_k \imath: A\otimes_k V_M\to A\otimes_k V_N$, for an inclusion $\imath:V_M\hookrightarrow V_N$ of graded vector spaces.  Any surjection of graded vector spaces splits. Denoting $S=\op{Ker} \phi$, as a graded vector space, $L\cong S\oplus K$. As graded vector spaces, 
$$\op{Hom}_{A/\psi}(N^{\odot n},L)\cong\op{Hom}_{k}({V_N}^{\odot n},S)\oplus A\langle {\psi}^n,[\psi^n,D^1]\rangle,$$
and $$\op{Hom}_{A/\psi\circ f}(M^{\odot n},L)\cong\op{Hom}_{k}({V_M}^{\odot n},S)\oplus A\langle {\psi}^n\circ f,[\psi^n\circ f,D^1]\rangle  .$$
The universal  map 
$$u:\op{Hom}_{A/\psi}({N}^{\odot n},L)\to\op{Hom}_{A/\psi\circ f}({M}^{\odot n},L)$$
under the above isomorphisms splits into the surjection $-\circ\imath$ and identity on the zero-disk.

Composing the right-hand side of (\ref{e1}) with ${f}^{\odot n}$ gives
$$-\frac{1}{2}\sum_{k=2}^{n-1}[g^k,D^{n-k+1}]=[{g^n},D^1].$$  The homotopy fiber
\begin{center}
\begin{tikzcd}
\op{Fib}\arrow[r]\arrow[d, two heads,"\sim"]&\op{Hom}_{A/\psi}(N^{\odot n},L)\arrow[d, two heads,"\sim"]\arrow[d, two heads,swap, "u"]\\
A\langle g^n,[g^n,D^1]\rangle   \arrow[r]&\op{Hom}_{A/\psi \circ f}(M^{\odot n},L)
\end{tikzcd}
\end{center}
is acyclic. As the left-hand side of (\ref{ld}) is a cycle in the homotopy fiber which maps by the vertical arrow to $[g^n,D^1]$, it is the image under the differential of a derivation $A$ to ${M}^{{\odot n}}$ denoted by ${l}^n$. Clearly, it satisfies equations (\ref{e2}) and (\ref{e3}).

\end{proof}
Next we move to general morphisms of SH LR[1] pairs. Fibrations of SH LR[1] pairs in ${\tt SHLR}^{\op{cof}}(k)$ are that of the category fibrant objects \cite{P}.
\begin{defi}
Lets $(A,M)$ and $(B,N)$ be SH LR[1] pairs with $M\in {\tt Mod}(A)$, and $N\in {\tt Mod}(B)$ cell complexes. A strict morphism $(f_0:B\to A, f^1:M\to A\otimes_BN)$ of SH LR[1] pairs is called
\begin{enumerate}
\item a \emph{cofibration} if $f_0$ is degreewise surjective, and $f^1$ a relative cell complex of $A$-modules; and
\item a \emph{fibration} if $f_0$ is a relative cell complex of cdga-s, and $f^1$ is a projection of cell complexes
$$f^1:(A  \langle v_i\rangle_{i\in I},d_V)\twoheadrightarrow(A \langle v_j\rangle_{j\in J},d_V)$$
with  $J\subset I$, and $f(v_i)=\begin{cases} v_i &\text{ for } i\in J\\ 0 &\text{ for } i\in I\setminus J\end{cases}$.
\end{enumerate}
\end{defi}
\begin{rem} All the results in the section remain valid if cofibrantions and fibrations are defined as retracts of the maps above. However, the proofs become more involved.
\end{rem}

We first construct a replacement of a SH LR[1] pair $(A,M)$ with arbitrary $A$ and a cell complex $M$ by an element in ${\tt SHLR}^{\op{cof}}(k)$. By \cite[Lemma 26]{P}, given a cdga $A$, a cell complex $M$ in the category of $A$-modules, and a cofibrant replacement $r_A:QA\to A$ in the category of cdga-s, there exists a cofibrant replacement by a cell complex $p_M:QM\to M$ in the category of $QA$-modules, such that $M=A\otimes_{QA} QM.$ 
\begin{lemma}\label{lemma20}
With $r_A:QA\to A$, and $p_M:QM\to M$ as above, there exists a SH LR[1] structure on $(QA,QM)$ such that $(r_A,\op{id}):(A,M)\to(QA,QM)$ is a fibration of SH LR[1] pairs.
\end{lemma}
\begin{proof}
As $QA$ is cofibrant, so is its module of K\"ahler differentials. Consequently, the functor $\op{Der}(QA,-)=\op{Hom}_{QA}(\Omega^1(QA),-)$ preserves trivial fibrations. Denoting by $(\Gamma^k_A)_k$ the anchor in $(A,M)$, anchor in $(QA,QM)$ is given by the lifting property in ${\tt Mod}(QA)$
\begin{center}
\begin{tikzcd}
&&\arrow[d,two heads, "\sim"] \op{Der}(QA,QA)[1]\\
QM^{\odot k}\arrow[urr,dashed,"\Gamma^k_{QA}"]\arrow[r,"p_M"]&M^{\odot k}\arrow[r,"\Gamma^k_{A}"]&\op{Der}(QA,A)[1].
\end{tikzcd}
\end{center}
It suffices to find an $L_{\infty}[1]$ algebra structure $(D^k_{QM})_k$ on $QM$ such that 
\begin{equation}\label{blablabla}
p_{M}\circ D^k_{QM}=D^k_{M}\circ p_{M}^{\odot k}.
\end{equation}

We proceed by induction. Assume $D^k$ have been defined for $k<n$.  
As graded modules, $M=A\langle m_i\rangle_{i\in I}$, and $QM=QA\langle m_i\rangle_{i\in I}$ for a well ordered set $I$. For each monomial $m_1\odot\cdots\odot m_n\in M^{\odot n}$, fix an $\overline{m}_{m_1\ldots i_n}\in QM$ such that
$p_M(\overline{m}_{m_1\ldots i_n})=D^n_M(m_1\odot\cdots\odot m_n),$
and define, for $\rho$ given by the Equation \ref{SHLRprop}, a provisional differential $\tilde{D}^n_{QM}:QM^{\odot n}\to QM$
by
\begin{equation*}
\begin{split}
\tilde{D}^{n}_{QM}&(a_1m_1,\ldots,a_{n}m_{n})=(-1)^{\sum_{i=1}^{n}|a_i|(1+|m_1|+\ldots+|m_{i-1}|)}a_1\cdots a_{n}\overline{m}_{m_1\dots m_{n}}\\
+&\sum_{k=1}^{n}(-1)^{{\sum_{i=1}^{n}|a_i|(1+|m_1|+\ldots+|m_{i-1}|)}+|m_k|(|m_{k+1}|+\ldots+|m_{n}|)+|a_k|(1+|a_{k+1}|+\ldots+|a_{n}|+|m_1|+\ldots+|m_{n}|)}\\
&a_1\cdots\hat{a}_k\cdots a_{n} \rho^k_{QA}(m_1\odot\cdots\odot \hat{m}_k\odot\cdots\odot m_{n})(a_k)\cdot m_k
\end{split}
\end{equation*}

Similarly as before,
\begin{equation}\label{aublablabla} [\tilde{D}^n_{QM},D^1_{QM}]+\frac{1}{2}\sum_{k=2}^{n-1}[D^k_{QM},D^{n-k+1}_{QM}] \end{equation}
is $A$-multilinear, and a cycle in $\op{Hom}_{QA}(QM^{\odot n}, QM),$ whose composition with $p_M$ is zero. Clearly, $\tilde{D}^{n}$ satisfies the condition (\ref{blablabla}). It remains to show that the expression (\ref{aublablabla}) is equal as a morphism of graded $QA$-modules to a boundary $[\Delta,D^1_{QM}]$ with $p\circ \Delta=0$. In this case ${D}^{n}_{QM}:=\tilde{D}^{n}_{QM}-\Delta$ is a correct extension of the $L{_\infty}$-structure to weight $n$.

As the (homotopy) fiber 
\begin{center}
\begin{tikzcd}
\op{Fib}\arrow[r,hook]\arrow[d, two heads,"\sim"]\arrow[d, two heads,swap, "p\circ-"]&\op{Hom}_{QA}(QM^{\odot n},QM)\arrow[d, two heads,"\sim"]\arrow[d, two heads,swap, "{p_M}\circ-"]\\
0   \arrow[r,hook]&\op{Hom}_{QA}(QM^{\odot n},M)
\end{tikzcd}
\end{center}
is acyclic, and the expression (\ref{aublablabla}) is its cycler, it is a boundary.

\end{proof}

Lifting properties happen to be satisfied in the general setting as well:
\begin{lemma}\label{lemma21}
Let $(A,L)$, $(B,M)$, $(C,N)$, and $(D,K)$ be SH LR[1] pairs. Assume that $D\in{\tt cdga}(k)$ is semi-free; that $L\in \op{Mod}(A)$, $M\in {\tt Mod}(B)$, $N\in {\tt Mod}(C)$ and $K\in {\tt Mod}(D)$ are cell complexes. Let $f:(B,M)\to (C,N)$ be a cofibration, and, $\phi:(A,L)\to(D,K)$ a fibration. Let $g:(B,M)\rightsquigarrow(A,L)$
and $\psi:(C,N)\to (D,K)$ be SH morphisms. If either $f$ or $\phi$ is a weak equivalence, there exists a SH morphism $l:(C,N)\rightsquigarrow(A,L)$ such that $g=l\circ f$, and $\psi=\phi\circ l$.
\begin{center}
\begin{tikzcd}
(B,M)\arrow[d, hook,"f"]\arrow[d, hook,swap,"\sim"]\arrow[rr,rightsquigarrow,"g"] &  &(A,L)\arrow[d, two heads,"\phi"] & &
(B,M)\arrow[d, hook,"f"]\arrow[rr,rightsquigarrow,"g"] &  &(A,L)\arrow[d, two heads,"\phi"]\arrow[d, two heads,swap,"\sim"]\\
(C,N)\arrow[urr,dashed,rightsquigarrow,"l"]\arrow[rr,rightsquigarrow,"\psi"]& &(D,K);& &
(C,N)\arrow[urr,dashed,rightsquigarrow,"l"]\arrow[rr,rightsquigarrow,"\psi"]& &(D,K).
\end{tikzcd}
\end{center}
\end{lemma}
\begin{proof}
Dually, construct $l^*$ in
\begin{center}
\begin{tikzcd}
 (\widehat{\op{Sym}}_D K^*)  \arrow[rr,rightsquigarrow,"\psi^*"]\arrow[d, hook,"\phi^*"]&  &(\widehat{\op{Sym}}_C N^*) \arrow[d, two heads,"f^*"]\arrow[d, two heads,swap,"\sim"]&&
 (\widehat{\op{Sym}}_D K^*)  \arrow[rr,rightsquigarrow,"\psi^*"]\arrow[d, hook,swap,"\sim"]\arrow[d, hook,"\phi^*"]&  &(\widehat{\op{Sym}}_C N^*) \arrow[d, two heads,"f^*"]\\
\arrow[rr,rightsquigarrow,"g^*"] (\widehat{\op{Sym}}_A L^*)\arrow[urr,dashed,rightsquigarrow,"l^*"]& & (\widehat{\op{Sym}}_B M^*),
&& \arrow[rr,rightsquigarrow,"g^*"] (\widehat{\op{Sym}}_A L^*)\arrow[urr,dashed,rightsquigarrow,"l^*"]& & (\widehat{\op{Sym}}_B M^*)
\end{tikzcd}
\end{center}
by induction on weight. In weight zero, restriction $l^*_0|_A:A\to C$, denoted by $l_0$, such that $ f_0\circ l_0=g_0$ and  $l_0\circ \phi_0=\psi_0$, is given by the lifting property in ${\tt cdga}(k)$. Restriction of $l^*_0$ to $L^*$, dually a morphism of $B$-modules $l^1:N\to C\otimes_A L$, is constructed within the step of induction as the special case $n=0$.

Assume $l^*$ has been defined up to weight $n-1$. 
As a semi-free cdga,  $D=(\op{Sym}_k V_D,d_D)$, for a graded vector space $V_D$. As a relative cell complex, $\phi_0$ is the inclusion $\op{Sym}_k V_D\hookrightarrow \op{Sym}_k V_A$, for an inclusion of graded vector spaces $V_D\subseteq V_A$

A map from a graded algebra to a pro-graded algebra is uniquely determined by its value in the limit, as
$$\op{Hom}_{\op{Pro}\mathcal{C}}(X,(Y_i)_{i\in I})=\varprojlim_{i\in I}\op{Hom}_{\mathcal{C}}(X,Y_i)=\op{Hom}_{\mathcal{C}}(X,\varprojlim_{i\in I}Y_i).$$ 
With this, any extension of $l^*$ to weight $n$ (for $n>0$), for which $\sum_{k\leq n}{l^*}_k$ remains multiplicative, requires the weight $n$-component ${l^*}_n|_A:A\to (\op{Sym}_C^n N)^*$ to be the sum of an arbitrary derivation valued in $(\op{Sym}_C^n N)^*$ and a fixed term completely determined by ${l^*_k}$, for $k<n$. Indeed, with the $A$-module structure on $(\op{Sym}_C^n N)^*$ determined by  $l_0:A\to C$,
\begin{equation*}
\begin{split}
l^*_n(v_1\cdots v_s)&=\sum_{k=1}^sv_1\cdots l^*_n(v_k)\cdots v_s+\sum_{\substack{n_1+\ldots+n_s=n\\n_k<n}}l^*_{n_1}(v_1)\cdots l^*_{n_s}(v_s)\\
&=:{l^*_n}^{\op{Der}}(v_1\cdots v_s)+{l^*_n}^{\epsilon}(v_1\cdots v_s).
\end{split}
\end{equation*}
Similarly,
$g^*_n|_A={g^*_n}^{\op{Der}}+{g^*_n}^{\epsilon}$, and $\psi^*_n|_A={{\psi^*_n}^{\op{Der}}}+{\psi^*_n}^{\epsilon}.$ 
Observe that $f^*\circ{l^*_n}^{\epsilon}={{g^*}}_n^{\epsilon}$, and ${l^*_n}^{\epsilon}\circ \phi_0={\psi^*_n}^{\epsilon}$. All that is left is to find ${l^*_n}^{\op{Der}}$, such that
\begin{equation}\label{ld}
[{l^*_n}^{\op{Der}},d_0]=-\sum_{k=1}^{n-1}[l_k^*,d_{n-k}]-[{l^*_n}^{\epsilon},d_0],
\end{equation}
\begin{equation}\label{lg}
{f^*}\circ{l^*_n}^{\op{Der}}={g^*_n}^{\op{Der}},
\end{equation}
and
\begin{equation}\label{lh}
{l^*_n}^{\op{Der}}\circ \phi_0={\psi^*_n}^{\op{Der}}.
\end{equation}

To shorten the notation, denote ${N^*}^{\hat{\odot}n}:=(\op{Sym}_C^n N)^*$. By the square zero property of the differential on $\op{Hom}_k(A,{N^*}^{\hat{\odot}n})$, $[[{l^*_n}^{\epsilon},d_0],d_0]=0$. Moreover,
\begin{equation*}
\begin{split}
\sum_{k=1}^{n-1}[[l_k^*,d_{n-k}],d_0]&=-\sum_{k=1}^{n-1}[[l_k^*,d_0],d_{n-k}]-\sum_{k=1}^{n-1}[d_{n-k},d_0],l_k^*]\\
&=\sum_{k=1}^{n-1}\sum_{l=1}^{k-1}[[l_l^*,d_{k-l}],d_{n-k}]+\frac{1}{2}\sum_{k=1}^{n-1}\sum_{l=1}^{n-k-1}[d_{l},d_{n-k-l}],l_k^*]\\
&=-\frac{1}{2}\sum_{k=1}^{n-1}\sum_{l=1}^{n-k-1}[d_{l},d_{n-k-l}],l_k^*]+\frac{1}{2}\sum_{k=1}^{n-1}\sum_{l=1}^{n-k-1}[d_{l},d_{n-k-l}],l_k^*]=0.
\end{split}
\end{equation*}
This shows that the right-hand side of the equation (\ref{ld}) is a cycle in $\op{Der}(A,{N^*}^{\hat{\odot} n})$.
Maps  $${f^*}\circ-:\op{Der}(A,{N^*}^{\hat{\odot} n})\to \op{Der}(A,{M^*}^{\hat{\odot} n})\text{, and }{f^*}\circ-:\op{Der}(D,{N^*}^{\hat{\odot} n})\to \op{Der}(D,{M^*}^{\hat{\odot} n})$$
are trivial fibrations of dg vector spaces in the first case, and fibrations in the second. Similarly, maps
$$-\circ\phi_0:\op{Der}(A,{N^*}^{\hat{\odot} n})\to \op{Der}(D,{N^*}^{\hat{\odot} n}),\text{ and }-\circ\phi_0:\op{Der}(A,{M^*}^{\hat{\odot} n})\to \op{Der}(D,{M^*}^{\hat{\odot} n})$$
are fibrations  in the first case, and trivial fibrations in the second. Denote by $k\langle {\psi^*}_n^{\op{Der}},[ {\psi^*}_n^{\op{Der}},d_0]\rangle$ the subcomplex of $\op{Der}(A,{M^*}^{\hat{\odot n}})$ generated by $ {\psi^*}_n^{\op{Der}}$ (and its image under the differential). Observe that the pullbacks 
\begin{center}
\begin{tikzcd}
\op{Der}_{/\psi^*}(A,{N^*}^{\hat{\odot n}})\arrow[r]\arrow[d, two heads]&\op{Der}(A,{N^*}^{\hat{\odot n}})\arrow[d, two heads,swap, "-\circ\phi_0"] & \op{Der}_{/f^*\circ\psi^*}(A,{M^*}^{\hat{\odot n}})\arrow[r]\arrow[d, two heads]&\op{Der}(A,{M^*}^{\hat{\odot n}})\arrow[d, two heads,swap, "-\circ\phi_0"]  \\
k\langle {\psi^*}_n^{\op{Der}},[{\psi^*}_n^{\op{Der}},d_0]\rangle   \arrow[r]&\op{Der}(D,{N^*}^{\hat{\odot n}}) & k\langle {f^*}\circ{\psi^*}_n^{\op{Der}},[{f^*}\circ{\psi^*}_n^{\op{Der}},d_0]\rangle   \arrow[r]&\op{Der}(D,{M^*}^{\hat{\odot n}})
\end{tikzcd}
\end{center}
are in fact homotopy pullbacks, and that the composition ${f^*}\circ-$ induces the map in the pullbacks which is a quasi-isomorphism. It is important to notice that it is also a fibration (i.e. degree-wise surjective). Namely, given 
$\xi\in \op{Der}_{/f^*\circ\psi^*}(A,{M^*}^{\hat{\odot n}})$, let ${\xi_1}\in \op{Der}(A,{N^*}^{\hat{\odot n}})$ be such that ${\xi_1}={f^*}\circ\xi.$ Define a derivation $\xi_2$ from $A=\op{Sym}_kV_A$ to ${N^*}^{\hat{\odot n}}$ on generators by
$$
\xi_2(v) = \begin{cases}
 \psi_n^{\op{Der}}(v)  & \text{for }v\in V_D \\
  \xi_1(v) & \text{otherwise}.
\end{cases}
$$
By construction, $\xi_2\in\op{Der}_{/\psi}(A,{N^*}^{\hat{\odot n}})$, and $\xi_2\circ\phi_0=\xi.$

 Composing the right-hand side of (\ref{ld}) with ${f^*}$ gives
$$[{g^*_n}^{\op{Der}},d_0]=\sum_{k=1}^{n-1}[g_k^*,d_{n-k}]-[{g^*_n}^{\epsilon},d_0].$$  The homotopy fiber
\begin{center}
\begin{tikzcd}
\op{Fib}\arrow[r]\arrow[d, two heads,"\sim"]&\op{Der}_{/\psi^*}(A,{N^*}^{\hat{\odot n}})\arrow[d, two heads,"\sim"]\arrow[d, two heads,swap, "{f^*}\circ-"]\\
k\langle {g^*}_n^{\op{Der}},[{g^*}_n^{\op{Der}},d_0]\rangle   \arrow[r]&\op{Der}_{/f^*\circ\psi^*}(A,{M^*}^{\hat{\odot n}})
\end{tikzcd}
\end{center}
is acyclic. As the left-hand side of (\ref{ld}) is a cycle in the homotopy fiber which maps by the vertical arrow to $[{g^*}_n^{\op{Der}},d_0]$, it is the image under the differential of a derivation from $A$ to ${N^*}^{\hat{\odot n}}$ denoted by ${l^*}_n^{\op{Der}}$, such that ${f^*}\circ{l^*}_n^{\op{Der}}={g^*}_n^{\op{Der}}$, and ${l^*_n}^{\op{Der}}\circ \phi_0={\psi^*_n}^{\op{Der}}.$

Let us now extend $l^* _n$ to $L^*$ , for $n\geq 0$, with $l^*_{<0}=0$. As  a fibration of SH LR[1] pairs, $\phi$ is dually the map graded pro-$A$-algebras
$$ \phi_0^*: \widehat{\op{Sym}}_D \langle\langle v^*_i \rangle\rangle_{i\in I}\to\widehat{\op{Sym}}_A \langle\langle v^*_j \rangle\rangle_{j\in J},\hspace{10pt} av^*_i\mapsto \phi_0(a)v^*_i.
$$
Any multiplicative extension of $l^*_{<n}$ to weight $n$ is a sum of an arbitrary $A$-linear derivation and a term completely determined by $l^*_{<n}$ and $l^*_n|_A$. Indeed, denote by 
$$({l^*_n}_{(i,k)}:\op{Sym}^{\leq \beta(k)}_AL^*_{\beta(i)}\to\op{Sym}^{\leq k}_CN^*_{i} )$$ a morphism of inverse systems representing such an extension $l^*_n$.
\begin{equation*}
\begin{split}
{l^*_n}_{(i,k)}(a\cdot v^*_1\cdots v^*_m)&=\sum_{k=1}^n av^*_1\cdots l^*_n(v^*_k)\cdots v^*_m+\sum_{\substack{k_0+\ldots+k_n=n\\k_i\neq n\text{ for }i>0}}l^*_{k_0}(a)\cdot {l^*}_{k_1}(v^*_1)\cdots l^*_{k_m}(v^*_m)\\
=&({{l^*_n}^{\op{Der}}}_{(i,k)}+{{l^*_n}^{\epsilon}}_{(i,k)})(a\cdot v^*_1\cdots v^*_m).
\end{split}
\end{equation*}
Morphisms of graded pro-algebras represented by $({{l^*_n}^{\op{Der}}}_{(i,k)})$ and $({{l^*_n}^{\epsilon}}_{(i,k)})$ are denoted respectively by  ${l^*_n}^{\op{Der}}$, and ${l^*_n}^{\epsilon}$. Clearly, ${l^*_n}={l^*_n}^{\op{Der}}+{l^*_n}^{\epsilon}$.

Similarly,
$g^*_n={{g^*}_n^{\op{Der}}}+{{g^*}_n^{\epsilon}},$ and ${\psi^*}_n={{\psi^*}_n^{\op{Der}}}+{{\psi^*}_n^{\epsilon}}$.
Observe that ${f^*\circ{l^*}}_n^{\epsilon}={{g^*}}^{\epsilon}$, and ${{l^*}}_n^{\epsilon}\circ\phi^*={{\psi^*}}_n^{\epsilon}$.
All that is left is to find ${{l^*}}_n^{\op{Der}}$, such that
\begin{equation}\label{ld2}
[{l^*_n}^{\op{Der}},d_0]=-\sum_{k=1}^{n-1}[l_k^*,d_{n-k}]-[{l^*_n}^{\epsilon},d_0],
\end{equation}
\begin{equation}\label{lg2}
{f^*}\circ{l^*_n}^{\op{Der}}={g^*_n}^{\op{Der}},
\end{equation}
and
\begin{equation}\label{lh2}
{l^*_n}^{\op{Der}}\circ \phi^*={\psi^*_n}^{\op{Der}}.
\end{equation}

Denote by $\op{Der}_{A,n}(\widehat{\op{Sym}}_AL^*, \widehat{\op{Sym}}_CN^*)$ the space of $A$-linear derivations of weight $n$. Observe that
\begin{equation}\label{obs}
\begin{split}\op{Der}_{A,n}(\widehat{\op{Sym}}_AL^*, \widehat{\op{Sym}}_CN^*)&=\varprojlim_{i,k}\varinjlim_{j,l}\op{Der}_{A,n}(({\op{Sym}}^{\leq l}_AL_j)^*, ({\op{Sym}}^{\leq k}_CN_i)^*)\\
&\cong\varprojlim_{i, k> n}\varinjlim_{j, l\geq k}\op{Der}_{A,n}(({\op{Sym}}^{\leq l}_AL_j)^*, ({\op{Sym}}^{\leq k}_CN_i)^*)
\\
&\cong\varprojlim_{i}\varinjlim_{j}\op{Hom}_{A}(L_j^*, (\op{Sym}^{n}_CN_i)^*)
\\
&\cong\op{Hom}_{A}(L^*, (\op{Sym}^{n}_CN)^*)\cong \op{Hom}_C({\op{Sym}}_C^{n}N,C\otimes_A L)
\end{split}
\end{equation}
Under the natural  isomorphism \ref{obs}, maps
\begin{equation*}
\begin{split}
&\phi^*\circ-:\op{Der}_A(\widehat{\op{Sym}}_AL^*, \widehat{\op{Sym}}_CN^*)\twoheadrightarrow \op{Der}_K(\widehat{\op{Sym}}_DK^*, \widehat{\op{Sym}}_CN^*)\\
&\phi^*\circ-:\op{Der}_A(\widehat{\op{Sym}}_AL^*, \widehat{\op{Sym}}_BM^*)\twoheadrightarrow \op{Der}_K(\widehat{\op{Sym}}_DK^*, \widehat{\op{Sym}}_BM^*)
\end{split}
\end{equation*}
are fibrations in the first case, and trivial fibrations in the second case, while the maps
\begin{equation*}
\begin{split}
&-\circ f^*:\op{Der}_A(\widehat{\op{Sym}}_AL^*, \widehat{\op{Sym}}_CN^*)\twoheadrightarrow \op{Der}_A(\widehat{\op{Sym}}_AL^*, \widehat{\op{Sym}}_BM^*)\\
&-\circ f^*:\op{Der}_A(\widehat{\op{Sym}}_AL^*, \widehat{\op{Sym}}_CN^*)\twoheadrightarrow \op{Der}_A(\widehat{\op{Sym}}_AL^*, \widehat{\op{Sym}}_BM^*)
\end{split}
\end{equation*} are trivial fibrations in the first case, and fibrations in the second case. Hence, the induced map between (homotopy) fiber products
\begin{center}
\begin{tikzcd}
\op{Der}_{A/\psi^*}(\widehat{\op{Sym}}_AL^*, \widehat{\op{Sym}}_CN^*)\arrow[r]\arrow[d, two heads,swap]&\op{Der}_A(\widehat{\op{Sym}}_AL^*, \widehat{\op{Sym}}_CN^*)\arrow[d, two heads,swap, "-\circ\phi^*"] \\
k\langle {\psi^*}_n^{\op{Der}},[{\psi^*}_n^{\op{Der}},d_0]\rangle   \arrow[r]&\op{Der}_D(\widehat{\op{Sym}}_DK^*, \widehat{\op{Sym}}_CN^*), 
\end{tikzcd}
\end{center}
and
\begin{center}
\begin{tikzcd}
 \op{Der}_{A/f^*\psi^*}(\widehat{\op{Sym}}_AL^*, \widehat{\op{Sym}}_BM^*)\arrow[r]\arrow[d, two heads]&\op{Der}_A(\widehat{\op{Sym}}_AL^*, \widehat{\op{Sym}}_BM^*)\arrow[d, two heads,swap, "-\circ\phi^*"] \\
 k\langle {f^*}{\psi^*}_n^{\op{Der}},[{f^*}{\psi^*}_n^{\op{Der}},d_0]\rangle   \arrow[r]&\op{Der}_D(\widehat{\op{Sym}}_DK^*, \widehat{\op{Sym}}_BM^*),
\end{tikzcd}
\end{center}
is a weak equivalence. It is also easily seen to be surjective. Finally, the right hand side of \ref{ld2} is a cycle (identically zero in the case $n=0$), thus a cycle in the homotopy fiber (of dg vector spaces)
\begin{center}
\begin{tikzcd}
\op{Fib}\arrow[r]\arrow[d, two heads,"\sim"]&\op{Der}_{A/\psi^*}(\widehat{\op{Sym}}_AL^*, \widehat{\op{Sym}}_CN^*)\arrow[d, two heads,"\sim"]\arrow[d, two heads,swap, "{f^*}"]\\
k\langle {g^*}_n^{\op{Der}},[{g^*}_n^{\op{Der}},d_0]\rangle   \arrow[r]&\op{Der}_{A/f^*\psi^*}(\widehat{\op{Sym}}_AL^*, \widehat{\op{Sym}}_BM^*),
\end{tikzcd}
\end{center}
which is mapped by the (vertical) trivial fibration to $[{g^*}_n^{\op{Der}},d_0]$. Hence, it must be the boundary $[{l^*}_n^{\op{Der}},d_0]$ of a derivation ${l^*}_n^{\op{Der}}$ which satisfies the equations (\ref{lg2}) and (\ref{lh2}).
\end{proof}

\section{BV-type resolutions}\label{BV}

We begin by retelling the story of the classical BV-BRST formalism (\cite{HT}) on an algebraic toy-model.
Let $$X=\mathbb{A}^n=\op{Spec}(k[x_1,\ldots x_n])$$
be our ambient space -- in the Langrangian formalism, it is the configuration space, and we call the variables $x_i$ fields. Let 
$$\Sigma=\op{Spec}(A)\hookrightarrow X$$
be a surface in $X$ -- call it critical surface of the Lagrangian.  Let $$\mathcal{F}\subset T_X$$ be a distribution -- called the space of infinitesimal Gauge symmetries. It happens that infinitesimal gauge symmetries restrict to $\Sigma$, and that the restriction
$$\mathcal{F}|_{\Sigma}\subset T_\Sigma= \op{Der}(A)$$ is involutive. Consequently, $(A,\mathcal{F}|_{\Sigma})$ is a LR pair. In the first step of BV formalism, we take a free resolution  of $\mathcal{F}|_{\Sigma}$ $$q:A\langle v_i^*\rangle\overset{\sim}{\twoheadrightarrow}\mathcal{F}|_{\Sigma}$$ inside the category of $A$-modules.  It has the structure of a SH LR algebra for which $q$  is a strict morphism of SH LR algebras (lemma \ref{lemma18}). Generators $v_i$ of its Chevalley-Eilenberg complex, dual to $v_i^*$, are called ghosts.
In the next step, we take a Koszul-Tate resolution of $A$ $$p_{\op{KT}}:(k[x_i,x^*_i,v^*_j],\delta)\overset{\sim}{\twoheadrightarrow} A$$ in the category of dg $\mathcal{O}_X$-algebras. Variables $x^*_i$ are called antifields, and variables $v^*_j$ are called antighosts. In the final step, we put together fields, ghosts, anti-fields, and anti-ghosts, to construct a SH LR pair whose CE complex $(k[x_i,x^*_i,v^*_j][[v_j]],s)$ 
 is called the BV complex. The map $(p_{\op{KT}},\op{id})$, dually,
 $$A\otimes_{k[x_i,x^*_i,v^*_j]}-:(k[x_i,x^*_i,v^*_j][[v_j]],s)\to (A[[{v_j}]],d)$$
 is a (strict) morphism of SH LR pairs (lemma \ref{lemma20}). 

In summary, the classical BV formalism replaces $(A,\mathcal{F}|_{\Sigma})$ by an object in ${\tt SHLR}^{\op{cof}}(k)$, obtained through a zig-zag of strict SH morphisms
\begin{center}
\begin{tikzcd}
&(A,\mathcal{F}|_{\Sigma})&\arrow[l, two heads, swap, "\sim"]\arrow[r,hook,"\sim"](A,A\langle v_i^*\rangle)&(k[x_i,x^*_i,v^*_j],k[x_i,x^*_i,v^*_j]\langle v_i^*\rangle).
\end{tikzcd}
\end{center}

In the Largangian field theory, fields are dual to anti-fields, and ghosts are dual to anti-ghost. Consequently, the BV complex carries a (-1)-shifted symplectic structure. As shifted symplectic structures go beyond the scope of the present paper, we must stop here with the wonderful story of BV-BRST formalism. Important for us is that the above procedure applies to any dg LR pair in place of an involutive distribution on a (possibly singular) surface, even if the resulting BV complex needs not be $(-1)$ -- shifted symplectic in general. We speak of BV resolutions of dg LR pairs, which are unique up to homotopy:
\begin{defi}
BV resolution of a dg LR pair $(A,M)$ is a SH quasi-isomorphism $(A,QM)\overset{\sim}{\rightsquigarrow}(QA,Q'M)$, with $QA$ a semi-free cdga, $QM$ a cell complex of $A$-modules, and $Q'M$ a cell complex of $QA$-modules, together with a linear SH quasi-isomorphism $\pi_{Q}:(A,QM)\xrightarrow{\sim} (A,M)$. Given a dg LR pair $(A,M)$, a morphism of its BV resolutions is a commutative diagram
\begin{center}
\begin{tikzcd}
\arrow[r,rightsquigarrow](A,Q_1M)\arrow[d,swap,"\sim"]\arrow[d,"f_0"]&(Q_1A,Q_1'M) \arrow[d,rightsquigarrow,"f_1"] \arrow[d,rightsquigarrow,swap,"\sim"]\\
(A,Q_2N)\arrow[r,rightsquigarrow]&(Q_2A,Q'_2M),
\end{tikzcd}
\end{center}
with $f_0$ a linear SH morphism, such that $\pi_{Q_2}\circ f_0=\pi_{Q_1}.$
\end{defi}
\begin{prop}
Given a dg LR pair $(A,M)$, the category of its BV resolutions has a simply connected nerve.
\end{prop}
\begin{proof}
First, we show that the nerve is non-empty. Denote by $\pi_u:Q_u M\twoheadrightarrow M$ a cofibrant replacement of $M$ by a cell complex in Nuiten's semi-model structure on dg LR algebras over $A$. Given a cofibrant replacement of $A$ by a semi-free cdga $r_A:\overline{Q}A\twoheadrightarrow A$, denote by $f_u: (A,Q_uM)\rightarrow(\overline{Q}A,{Q}'_u M)$ a strict morphism of SH LR pairs given by the Lemma \ref{lemma20}. Clearly, $f_u$ is a BV resolution. 

To prove that the nerve is connected, given any other BV resolution $f_Q:(A,QM)\rightsquigarrow(QA,Q'M)$ of $(A,M)$, we construct a zig-zag of morphisms of BV resolutions 
\begin{center}
\begin{tikzcd}
f_Q&f_{\overline{Q}}\arrow[l,swap,"g"]\arrow[r,"h"]&f_u.
\end{tikzcd}
\end{center}
Denote by $f_{\overline{Q}}: (A,QM)\rightarrow(\overline{Q}A,\overline{Q}' M)$ a strict morphism of SH LR pairs given by the Lemma \ref{lemma20}. $g_1:(\overline{Q}A,\overline{Q}' M)\to (QA,Q'M)$ is the lift
\begin{center}
\begin{tikzcd}
(A,QM)\arrow[d, hook,"f_{\overline{Q}}"]\arrow[d, hook,swap,"\sim"]\arrow[d, hook,"f_{\overline{Q}}"]\arrow[rr,rightsquigarrow,"f_Q"] &  & (QA,Q'M)\arrow[d, two heads] \\
(\overline{Q}A,\overline{Q}' M)\arrow[urr,dashed,rightsquigarrow,"g_1"]\arrow[rr,rightsquigarrow]& &(k,0).
\end{tikzcd}
\end{center}
With this, we have constructed the morphism $g=(\op{id},g_1):f_{\overline{Q}}\to f_Q$. It remains to construct $h:f_{\overline{Q}}\to f_{u}$. Linear SH morphism $h_0:{Q}M\to Q_u M$ of SH LR $A$-algebras is given by the lifting property
\begin{center}
\begin{tikzcd}
0\arrow[d, hook]\arrow[d, hook]\arrow[rr]  & &Q_uM\arrow[d, two heads,"\pi_u"]\arrow[d,two heads,swap,"\sim"]\\
{Q}M\arrow[urr,dashed,"h_0"]\arrow[rr,"\pi_{\overline{Q}}"] & &M,
\end{tikzcd}
\end{center}
and the SH LR morphism $h_1:(\overline{Q}A,\overline{Q}'M)\to (\overline{Q}A,Q_u'M)$ is given by the lifting property
\begin{center}
\begin{tikzcd}
(A,QM)\arrow[d, hook,"f_{\overline{Q}}"]\arrow[d, hook,swap,"\sim"]\arrow[d, hook,"f_{\overline{Q}}"]\arrow[rr,rightsquigarrow,"f_u\circ h_0"] &  & (\overline{Q}A,Q'_uM)\arrow[d, two heads] \\
(\overline{Q}A,\overline{Q}'M)\arrow[urr,dashed,rightsquigarrow,"h_1"]\arrow[rr,rightsquigarrow]& &(k,0).
\end{tikzcd}
\end{center}

Finally, we show that the nerve has no non-trivial cycles. First, we prove that any two morphisms $h$ as above are equal in the homotopy category of BV resolutions. Assume $h^1$ and $h^2$ to be such maps. $Q_u M$ being a dg LR algebra over $A$, morphisms $h_0^1$ and $h_0^2$ factor through a map $\mathcal{L}_AC_A(QM)\coprod \mathcal{L}_AC_A(QM)\to Q_u M$. Denote by $\op{Cyl}$ the cylinder in Nuiten's semi-model structure on the category of dg LR algebras over $A$. We get the lifting diagram
\begin{center}
\begin{tikzcd}
\mathcal{L}_AC_A(QM)\coprod \mathcal{L}_AC_A(QM)\arrow[d, hook,"{(\imath_1,\imath_2)}"]\arrow[rr]  & &Q_uM\arrow[d, two heads,"\pi_u"]\arrow[d,two heads,swap,"\sim"]\\
\op{Cyl}(\mathcal{L}_AC_A(QM))\arrow[urr,dashed,swap,"H"]\arrow[rr] & &M.
\end{tikzcd}
\end{center}
For $(B,N)\in {\tt SHLR}^{\op{cof}}(k)$, denote by $\op{Path}(B,N)$ the path object from \cite{P}. Given a dg LR pair $(A,L)$ with $L$ a cell complex of $A$-modules, left homotopic maps $(A,L)\to (B,N)$ are also right homotopic: the proof of Proposition 1.2.5(v) in \cite{Ho} carries over verbatim. Hence, we obtain a right homotopy
\[
\begin{tikzcd}[column sep=large]
& (A,QM)
  \ar[d,"\eta_A"]
  \ar[ddl,swap,"h^1_2\circ  f_{\overline{Q}}", bend right=25, pos=0.7]
  \ar[ddr," h^2_2\circ  f_{\overline{Q}}", bend left=25, pos=0.7]
& \\[2pt]
& (A,\mathcal{L}_AC_A(QM))
  \ar[dl,swap,"f_u\circ H\circ\imath_1", pos=0.6]
  \ar[d,"K"]
  \ar[dr,"f_u\circ H\circ\imath_2", pos=0.6]
& \\[2pt]
(\overline{Q}A,Q'_uM)
& \op{Path}(\overline{Q}A,Q'_uM)
  \arrow[l,two heads]
  \arrow[r,two heads]
& (\overline{Q}A,Q'_uM).
\end{tikzcd}
\]
By the lifting property, we conclude that maps $h_1^1$ and $h_1^2$ are homotopic as well
\begin{center}
\begin{tikzcd}
(A,QM)\arrow[d, hook,"f_{\overline{Q}}"]\arrow[d, hook,swap,"\sim"]\arrow[rr,rightsquigarrow," K\circ \eta_A"] &  &\op{Path}(\overline{Q}A,Q'_uM)\arrow[d, two heads] \\
(\overline{Q}A,\overline{Q}'M)\arrow[urr,dashed,rightsquigarrow,"I"]\arrow[rr,rightsquigarrow,"{(h_2^1,h_2^2)}"]& & (\overline{Q}A,Q'_uM)\times  (\overline{Q}A,Q'_uM).
\end{tikzcd}
\end{center}
With $J$ as in the proof of \cite[Proposition 1.2.5(v)]{Ho}, and $i=1,2$, we define morphisms of BV resolutions $\alpha$ and $\beta_i$ such that $\beta^1\circ \alpha=h^1$, and $\beta^2\circ \alpha=h^2$:
\begin{center}
\begin{tikzcd}
(A,QM)\arrow[d, "\imath_1\circ \eta_A"]\arrow[rr,rightsquigarrow,"f_{\overline{Q}}"] &  &(\overline{Q}A,\overline{Q}'M)\arrow[d, rightsquigarrow,"I"] \\
\op{Cyl}(A,\mathcal{L}_AC_AQM)\arrow[d,"H"]\arrow[rr,rightsquigarrow,"J"] &  &\op{Path}(\overline{Q}A,Q'_uM)\arrow[d, two heads,"p_i"] \\
(A,Q_uM)\arrow[rr,rightsquigarrow,"f_u"]& & (\overline{Q}A,Q'_uM).
\end{tikzcd}
\end{center}
commutes. Denote by $\beta$ the morphisms of BV resolutions
\begin{center}
\begin{tikzcd}
\op{Cyl}(A,\mathcal{L}_AC_AQM)\arrow[d, equal]\arrow[rr,rightsquigarrow,"f_u\circ H"] &  &(\overline{Q}A,Q'_uM)\arrow[d, hook] \\
\op{Cyl}(A,\mathcal{L}_AC_AQM)\arrow[rr,rightsquigarrow,"J"] &  &\op{Path}(\overline{Q}A,Q'_uM).
\end{tikzcd}
\end{center}
Clearly, $\beta^1\circ\beta=\beta^2\circ\beta$. Hence $\beta^1=\beta^2$, and finally $h^1=h^2$.

Any two different BV resolutions $f_{Q_1}$, $f_{Q_2}$ are connected by the path $(g^2)\circ (h^2)^{-1} \circ (h^1) \circ (g^1)^{-1}.$ To prove that the nerve is simply connected it suffices to show that any morphism of BV resolutions $\phi: f_{Q^1}\to f_{Q^2}$ is homotopic to that path. We do this by constructing a morphism $\overline{\phi}:f_{\overline{Q}^1}\to f_{\overline{Q}^2}$ such that the diagram
\begin{equation}\label{diagbv}
\begin{tikzcd}
&f_u&\\
f_{\overline{Q}^1}\arrow[d,"g^1"]\arrow[rr,dashed,"\overline{\phi}"]\arrow[ur,"h^1"]&&f_{\overline{Q}^2}\arrow[d,"g^2"]\arrow[ul,swap,"h^2"]\\
f_{Q^1}\arrow[rr,"\phi"]&&f_{Q^2}
\end{tikzcd}
\end{equation}
commutes in the homotopy category of BV resolutions. By definition, $\overline{\phi}_0=\phi_0$, and $\overline{\phi}_1$ is given by the lifting property
\begin{center}
\begin{tikzcd}
(A,{Q}^1 M)\arrow[d,hook,"f_{\overline{Q}^1} "]\arrow[rr,"f_{\overline{Q}^2}\circ \overline{\phi}_0 "]& &(\overline{Q}A,\overline{Q}^2 M).\\
(\overline{Q}A,\overline{Q}^1 M)\arrow[urr, dashed,"\overline{\phi}_1"]& & 
\end{tikzcd}
\end{center}

It remains to prove that the diagram (\ref{diagbv}) commutes. By uniqueness, $h^1=h^2\circ \overline{\phi}$.
Denote by 
$$(QA,QM)\overset{i_{Q}}{\rightarrow}\op{Path}(QA,QM)\overset{p^1_{Q}\times p^2_{Q}}{\twoheadrightarrow}(QA,QM)\times(QA,QM)$$
the path decomposition in ${\tt SHLR}^{\op{cof}}(k)$. Given a BV resolution $f_Q:(A,QM)\rightsquigarrow(QA,Q'M)$, denote by $f_{\op{Path}}(Q)$ the BV resolution $ i_{Q'}\circ f_Q:(A,QM)\rightsquigarrow\op{Path}(QA, Q'M)$. This yields the canonical commutative diagram of BV resolutions
\begin{center}
\begin{tikzcd}
& f_Q\arrow[d,"i_{f_Q}"]\arrow[dl,swap,"\op{id}"]\arrow[dr,"\op{id}"] & \\
f_Q & f_{\op{Path}}(Q) \arrow[l,"p^1_{f_Q}"]\arrow[r,swap,"p^2_{f_Q}"]& f_Q.
\end{tikzcd}
\end{center}
Consequently, morphisms $p^1_{f_Q}$ and $p^2_{f_Q}$ are equal in the homotopy category of BV resolutions. The map $H$ in the lifting diagram
\begin{center}
\begin{tikzcd}
(A,{Q}^1 M)\arrow[d,hook,swap,"\sim"]\arrow[d,hook," f_{\overline{Q}^1} "]\arrow[rrr,"i_{f_{Q_ 2}}\circ f_{Q^2}\circ {\phi}_0"]&  & &\op{Path}(Q^2A,{Q^2}'A) \arrow[d,two heads," p_{f_{{Q^2}}}"]\\
(\overline{Q}A,{{\overline{Q}}^1}' M)\arrow[rrr,swap,"{(g ^2_1\circ {\overline{\phi}}_1,\phi_1\circ g^1_1)}"]\arrow[urrr, dashed,"H"]&  & & (Q^2A,{Q^2}'A)\times (Q^2A,{Q^2}'A).
\end{tikzcd}
\end{center}
yields a morphism of BV resolutions $(\phi_0,H):f_{\overline{Q}^1}\to f_{\op{Path}}(Q^2)$ such that $p^1_{f_{Q^2}}\circ H =g^2\circ \overline{\phi}$, and $p^2_{f_{ Q^2}}\circ H =\phi\circ g^1$. Thus, in the homotopy category, $g^2\circ \overline{\phi}=\phi\circ g^1.$ 
\end{proof}

A question one may ask is how well the BV resolutions behave with respect to morphisms, i.e. does a morphism of dg LR pairs $(f_0,f):(A,M)\to (B,N)$ induce a morphism of BV resolutions which is unique up to homotopy? However, the morphism $(f_0,f)$ is homotopy coherent only if $A\otimes_B N$ is of the correct homotopy type. Thus the problem should be restricted either
\begin{enumerate}
\item to morphisms of SH LR algebras over a fixed cdga $A=B$, or
\item to morphisms of SH LR pairs $(A,M)$ with $M$ a cofibrant $A$-module.
\end{enumerate}
First problem is resolved in \cite[Theorem 2.4]{LLG}. The second problem has a similar answer:
\begin{prop}\label{pr}
Let $(A,M)$ and $(B,N)$ be SH LR pairs with semi-free $B$ and cell complexes $M\in{\tt Mod}(A)$, $N\in {\tt Mod}(B)$. Let $f:(A,M)\to (B,N)$ be any SH morphism, and let $i: (A,M)\hookrightarrow (QA,\overline{M})$ be a trivial cofibration.
\begin{enumerate}
\item There exists a SH morphism (called lift) $l:(QA,\overline{M})\to (B,N)$ such that $l\circ i=f$.
\item Lift $l$ is unique up to homotopy.
\end{enumerate}
\end{prop}
\begin{proof}
Part 1 is a special case of Lemma \ref{lemma21}. For part 2, suppose $l_1$ and $l_2$ are two such lifts. Applying lemma \ref{lemma21} to the commutative diagram
\begin{center}
\begin{tikzcd}
(A,M)\arrow[d, hook,"i"]\arrow[d, hook,swap,"\sim"]\arrow[rr,"\imath_{\op{Path}}\circ f"] &  &\op{Path}(B,N) \arrow[d,two heads, "p_{\op{Path}}"]\\
(QA,\overline{M})\arrow[urr,dashed,"H"]\arrow[rr, "(l_1l_2)"]& &(B,N)\times (B,N),
\end{tikzcd}
\end{center}
we get the desired homotopy.
\end{proof}
\begin{cor}\label{GZ} Under 1-categorical Gabriel-Zisman localization, the category ${\tt SHLR}^{\op{cof}}(k)$ is equivalent to the category of SH LR pairs $(A,M)$ with arbitrary $A\in{\tt cdga}(k)$ and a cell complex $M\in {\tt Mod}(A)$.
\end{cor}

\section{Cartesian fibration of dg Lie-Rinehart algebras}\label{cart}
In the final section, we restrict our attention to dg LR algebras over cdga-s of finite type, that is, to dg LR pairs $(A,M)$ where $A$ is a finite cell complex -- a semi-free cdga on finitely many generators. The category of dg LR pairs which satisfy this finiteness condition is denoted by ${\tt dgLR}^{\op{ft}}(k)$. Its full subcategory of cofibrant objects -- pairs $(A,M)$ such that $A$ a is a finite cell complex, and $M$ a cell complex in Nuiten's semi-model structure on dg LR algebras over $A$ -- is denoted by ${\tt dgLR}^{\op{ft,cof}}(k)$. Its full subcategory of Nuiten fibrant-cofibrant objects is denoted by ${\tt dgLR}^{\op{ft,fib,cof}}(k)$.

Main issue is that the pullback of dg LR algebras is somewhat ill-behaved: one one side, when applied to arbitrary dg LR algebras, it does not respect weak equivalences, while on the other side, pullback of a fibrant-cofibrant dg LR algebra in the Nuiten's semi-model structure does, but is itself not a fibrant-cofibrant dg LR algebra. Nice LR algebras are defined in a manner that their pullbacks respect quasi-isomorphisms, and are themselves nice LR algebras.

\begin{defi}
A dg Lie-Rinehart pair $(A,M)$ is \emph{nice} if
\begin{itemize}
\item $A \in {\tt cdga}(k)$ is a finite cell complex;
\item anchor $\Gamma_M:M\to\op{Der}(A)$ is surjective;
\item $M$ is projective as a graded $A$ module; and
\item tensoring $-\otimes_A M$ preserves quasi-isomorphism.
\end{itemize}
The full subcategory of ${\tt dgLR}^{\op{ft}}(k)$ consisting of nice LR pairs is denoted by  ${\tt dgLR}^{\op{nice}}(k)$.
\end{defi}
As promised:
\begin{lemma}
Pullback of a nice LR pair is a nice LR pair.
\end{lemma}
\begin{proof}
Given a morphism of finite type cdga-s $f:A\to B$, and a nice LR pair $(B,M)$, its pullback $(A,f^! M)$ is the (homotopy) pullback
\begin{center}
\begin{tikzcd}
f^! M\arrow[rr]\arrow[d,two heads,"\Gamma_{f^! M}"]&&M\otimes_A B\arrow[d,two heads,"\Gamma_M\otimes_A B"]\\
\op{Der}(B)\arrow[r]&\op{Der}(A,B)\arrow[r, equal]&\op{Der}(A)\otimes_AB.
\end{tikzcd}
\end{center}
Clearly, the anchor is surjective. As graded $B$-modules, both $M\otimes_A B$ and $\op{Der}(A)\otimes_AB$ are projective. Hence the graded $B$-module
\begin{equation}\label{pullback}f^! M\cong \op{Der}(B)\oplus \op{Ker}(\Gamma_M)\otimes_A B\end{equation}
is also projective. It remains to prove that the tensor product $-\otimes_B f^! M$ preserves quasi-isomorphisms. Given a $B$-module $N$, the tensor product $f^! M\otimes_B N$ is the pullback
\begin{center}
\begin{tikzcd}
f^! M\otimes_B N\arrow[r]\arrow[d,two heads,"\Gamma_{f^! M}\otimes_B N"]&M\otimes_A B \otimes_B N\arrow[d,two heads,"\Gamma_M\otimes_A B \otimes_B N"]\\
\op{Der}(B)\otimes_B N\arrow[r]&\op{Der}(A,B)\otimes_B N.
\end{tikzcd}
\end{center}
A quasi-isomorphism $N\to N'$ induces a weak equivalence of the cospan diagrams
$$M\otimes_A B \otimes_B N\twoheadrightarrow\op{Der}(A,B)\otimes_B N\leftarrow \op{Der}(B)\otimes_B N$$
and
$$M\otimes_A B \otimes_B N'\twoheadrightarrow\op{Der}(A,B)\otimes_B N'\leftarrow \op{Der}(B)\otimes_B N'.$$
The map
 $f^! M\otimes_B N\to f^! M\otimes_B N'$ is the induced quasi-isomorphism between their (homotopy) pullbacks.
\end{proof}
When applied to nice dg LR pairs, pullback is homotopy coherent:
\begin{prop}\label{propalica}
The functor $$\pi:{\tt dgLR}^{\op{nice}}(k)\to{\tt gdca}^{\op{ft}}(k)^{\op{op}},\hspace{10pt}(A,M)\mapsto A$$ is a Cartesian fibration.
\end{prop}
\begin{proof}
Given $(A,M)\in {\tt dgLR}^{\op{nice}}(k)$, and a morphism $f:A\to B$ in ${\tt gdca}^{\op{ft}}(k)$, the universal map $(f,p):(B, f^!M)\to (A,M)$ is locally Cartesian (Proposition \ref{fact1}). To conclude that $\pi$ is a Grothendieck fibration (of 1-categories), it remains to show that, given maps $f:A\to B$, $g:B\to C$, and a nice dg LR algebra $M$ over $A$,  $g^!(f^!M)=(g\circ f)^!M$. However, this is an immediate consequence of the isomorphism \ref{pullback}. 

Denote by $N$ the nerve functor, by $V$ weak equivalences in ${\tt dgLR}^{\op{nice}}(k)$, and by $W$ weak equivalences in ${\tt gdca}^{\op{ft}}(k)^{\op{op}}$. To prove that $\pi$ is a Cartesian fibration, it suffices, by \cite[Proposition 2.1.4.]{H}, to show that the functor of marked simplicial sets
$$\pi:(N({\tt dgLR}^{\op{nice}}(k)),N(V))\to (N({\tt gdca}^{\op{ft}}(k)^{\op{op}}),N(W))$$
is a marked Cartesian fibration in the sense of \cite[Definition 2.1.1.]{H}. Defining property (1) holds because $\pi$ is a Grothendieck fibration; properties (2) and (3) are immediate. For the property (4), observe that, objects in ${\tt gdca}^{\op{ft}}(k)$ being fibrant-cofibrant, for every quasi-isomorphism $f:A\to B$, there exists another quasi-isomorphism $g:B\to A$. Universal maps $(g\circ f)^!M\to M$ and $(f\circ g)^!M\to M$ are  natural weak equivalences of nice dg LR algebras, respectively over $A$ and $B$. Hence, by  \cite[Definition 3.3, Proposition 7.5]{Barwick-Kan}, functors $g^!$ and $f^!$ form an equivalence of infinity categories between nice dg LR algebras over $A$ and $B$, as required.
\end{proof}

We are finally prepared to prove the main result of the section:
\begin{theo}\label{cartfib}
    The functor
    $${\tt dgLR}^{\op{ft, cof}}(k)\to {\tt cdga}^{\op{ft}}(k)^{\op{op}},\hspace{10pt}(A,M)\mapsto A$$
    decomposes into an equivalence of infinity categories ${\tt dgLR}^{\op{ft, cof}}(k)\cong{\tt dgLR}^{\op{nice}}(k)$ followed by the Cartesian fibration provided by the proposition \ref{propalica}. 
\end{theo}
\begin{proof}
Denote by $R(A,0)$ the dg LR pair obtained by applying the functor $\mathcal{L}_AC_A$ to the  SH LR algebra $(A,\overline{0})$ (defined in equation \ref{ass}). Clearly, $R(A,0)$ is Nuiten fibrant-cofibrant, and the map of dg LR pairs $(\op{id}_A,0):(A,0)\to R(A,0)$, which is the composition of the unit in \ref{adjunkcija2} with that in \ref{main_adj}, is a weak equivalence. Hence, given a cell complex $M$ in the semi-model category of dg LR algebras over $A$, the coproduct $$(A,M)\hookrightarrow (A,M)\coprod_{(A,0)}R(A,0)$$ is a cofibrant-fibrant replacement by a cell complex. Importantly, in contrast with the cofibrant-fibrant replacement guaranteed by the semi-model structure, it is functorial in ${\tt dgLR}^{\op{ft,cof}}(k)$, that is over a variable base. Hence it is an equivalence of $\infty$-categories
$${\tt dgLR}^{\op{ft,cof}}(k)\cong{\tt dgLR}^{\op{ft,fib,cof}}(k).$$

It still remains to prove that the inclusion ${\tt dgLR}^{\op{ft,fib,cof}}(k)\hookrightarrow {\tt dgLR}^{\op{nice}}(k)$ is also an equivalence of $\infty$-categories. Here, we follow the approach  of \cite[Proposition 1.3.4]{H}. The cited result is that in the presence of a model category structure without a functorial cofibrant replacement, inclusion of (fibrant)-cofibrant objects into the category (of fibrant objects) is still an equivalence of $\infty$-categories. As it turns out, all the model-categorical constructions within the proof are valid in our setting. In fact, the proof carries over verbatim; all that is needed are the following properties, which are immediate.
\begin{lemma}\begin{enumerate}
\item
Given nice dg LR algebras $L,M$, and $N$ over $A$, and a span diagram $L\to M \twoheadleftarrow N$, the product $L\times_M N$ is again a nice dg LR algebra over $A$.
\item Given a fibration $g:M\to N$ of nice dg LR algebras over $A$, and a morphism $g:A\to B$ in ${\tt cdga}^{\op{ft}}(k)$, the map
$f^!g:f^!M\to f^! N$ is a fibration of nice dg LR algebras over $B$.
\end{enumerate}
\end{lemma}
\begin{lemma}
Denote by $[n]$ the category whose objects are the numbers $0,1,\ldots, n$, with a unique morphism $i\to j$ whenever $i\leq j$. Denote by $({\tt dgLR}^{\op{nice}}(k))^{[n]}_{\op{const}}$ the full subcategory of the functor category $({\tt dgLR}^{\op{nice}}(k))^{[n]}$, whose morphisms, denoted by $p:(A,M)\to (A,N)$, are objectwise identity on the cdga-component
\begin{center}
\begin{tikzcd}[column sep=large]
(A_0,M_0)\ar{r}{(f_0,g_0)}\ar{d}{(\op{id}_{A_0},p_0)}&\cdots\ar{r}{(f_{n-1},g_{n-1})}&(A_n,M_n)\ar{d}{(\op{id}_{A_n},p_n)}\\
(A_0,N_0)\ar{r}{(f_0,h_0)}&\cdots\ar{r}{(f_{n-1},h_{n-1})}&(A_n,N_n).
\end{tikzcd}
\end{center}
A map  $p:(A,M)\to (A,N)$ in $({\tt dgLR}^{\op{nice}}(k))^{[n]}_{\op{const}}$ is called \emph{fibration} if for all $k$, the universal morphism
$$M_k\to N_k\times_{f_{k}^! N_{k+1}}f_{k}^!M_{k+1}$$
is a fibration in the semi-model structure on dg LR algebras over $A_k$. A fibration which is also a objectwise weak equivalence is called a \emph{trivial fibration}. The following results hold in $({\tt dgLR}^{\op{nice}}(k))^{[n]}_{\op{const}}$:
\begin{enumerate}
\item Every morphism factors into an objectwise cofibration followed by a trivial fibration.
\item Given a morphism $f:(A,M)\to (A,N)$ with $(A,M)$ objectwise cofibrant, and a trivial fibration $g:(A,QN)\to(A,N)$, there is a morphism $l:(A,M)\to (A,QN)$ such that $f=g\circ l$.
\end{enumerate}
\end{lemma}
\end{proof}

\appendix
\section{Differential graded pro-algebras} \label{Ap}
In the appendix, we recall the preliminary results from \cite{P}.

A partial ordered set $I$ is said to be \emph{cofiltered} if it is non-empty; and for every two elements $i,j\in I$ there exists an element $k\in I$ such that $k \geq i$, $k\geq j$. Inverse system in a category $\mathcal{C}$ is a functor $X: I\to \mathcal{C}$ for a cofiltered set $I$, often denoted by $(X_i)_{i\in I}$. For $i\geq j $, associated morphism $X_i\to X_j$ is denoted by $p_{ij}$.
The category $\op{Pro}(\mathcal{C})$ has as objects inverse systems in $\mathcal{C}$, and morphisms are
$$\op{Hom}_{\op{Pro}(\mathcal{C})}(X,Y)=\varprojlim_i\varinjlim_j \op{Hom}_{\mathcal{C}}(X_j,Y_i).$$
Equivalently, morphisms are equivalence classes of morphisms of inverse systems (\cite{MS}):

Given inverse systems $(X_\alpha,p_{\alpha,\alpha'})$ and $(Y_\beta,p_{\beta,\beta'})$  in $\mathcal{C}$, morphism of inverse systems $(X_\alpha)\to(Y_\beta)$ consists of a function $\phi:B\to A$ and of morphisms $f_\beta:X_{\phi(\beta)}\to Y_\beta,$ one for each $\beta\in B$, such that for all $\beta,\beta'\in B$, $\beta>\beta'$ there exists $\alpha>\phi(\beta),\phi(\beta')$ for which the diagram
\begin{center}
\begin{tikzcd}
X_{\phi(\beta)}\arrow[d,"f_\beta"]&\arrow[l,swap,"p_{\alpha,\phi(\beta)}"]X_{\alpha}\arrow[r,"p_{\alpha,\phi(\beta')}"]&X_{\phi(\beta')}\arrow[d,"f_{\beta'}"]\\
Y_{\beta}\arrow[rr,"p_{\beta,\beta'}"]&&Y_{\beta'}
\end{tikzcd}
\end{center}
commutes. Morphisms of inverse systems are denoted by $(f_\beta,\phi).$ We say that morphisms of inverse systems $(f_\beta,\phi)$ and $(g_{\beta},\psi)$ are equivalent if for every $\beta\in B$ there exists $\alpha\in A$, $\alpha>\phi(\beta),\psi(\beta)$ such that the diagram
\begin{center}
\begin{tikzcd}
X_{\phi(\beta)}\arrow[dr,"f_\beta"]&\arrow[l,swap,"p_{\alpha,\phi(\beta)}"]X_{\alpha}\arrow[r,"p_{\alpha,\psi(\beta)}"]&X_{\psi(\beta)}\arrow[dl,"g_{\beta}"]\\
&Y_{\beta}&
\end{tikzcd}
\end{center}
commutes.

Given a differential graded graded-commutative $k$-algebra $A$, and a morphism of graded $A$-modules $f:M\to N$, $f$-differential, denoted by $\op{Diff}_f(M,N)$ is a $k$-linear degree $1$ map $d:M\to N$ such that
$$d(am)=d(a)f(m)+(-1)^{|a|}ad(m).$$ Space of $f$-differentials is denoted by $\op{Diff}_f(M,N).$ Finally, a differential graded pro-module over $A$ is a pro object $(M_i)_{i\in I}$ in the category of graded $A$-module, together with a square zero differential
$$d\in \varprojlim_i\varinjlim_j \op{Diff}_{p_{ij}}(M_j,N_i).$$
Morphism of differential graded pro-modules is a map $f$ of graded pro-modules which respects the differentials.

Similarly, given a morphism $f:A\to B$ of differential graded graded-commutative $k$-algebras, $f$-differential, denoted by $\op{Diff}_f(A,B)$ is $k$-linear degree $1$ map $d:M\to N$ such that
$$d(ab)=d(a)f(b)+(-1)^{|a|}f(a)d(b).$$ Differential graded pro-algebra is a pro object in the category of graded commutative algebras $(A_i)$ together with a square zero differential
$$d\in \varprojlim_i\varinjlim_j \op{Diff}_{p_{ij}}(A_j,B_i).$$ Morphisms of dg-pro algebras are defined likewise.

By Lazard's theorem (or better its slight generalization to graded modules),  given a graded-commutative algebra $A$, the category of flat graded $A$-modules is equivalent to the ind-category of free finite $A$-modules, which is by dualization equivalent to the opposite of the pro-category of free finite $A$-modules. Namely, every flat module is an injective limit of free finite modules, so
$$\op{Hom}_{\tt Mod (A)}(\varinjlim_i {\tt M}_i, \varinjlim_j {\tt N}_j)=\varprojlim_i\op{Hom}_{\tt Mod (A)}( {\tt M}_i, \varinjlim_j {\tt N}_j)=\varprojlim_i\varinjlim_j\op{Hom}_{\tt Mod (A)} ({\tt M}_i,  {\tt N}_j)=\varprojlim_i\varinjlim_j\op{Hom}_{\tt Mod (A)} ({\tt N}_j^*,{\tt M}_i ^*).$$ Consequently, the category of differential graded-flat $A$-modules is equivalent to the opposite category of differential graded pro-finite $A$-modules. In this sense, we say that all graded-flat modules are dualizable.

Given a graded-flat $A$-module $M=\varinjlim_i M_i$,  the graded commutative algebra $$\op{Sym}^{\leq k}M_i^*:=\op{Sym}_AM_i^*/\op{Sym}^{>k}M_i^*$$  is a finite extension of $A$. The completed symmetric algebra on $M^*$ is a graded-commutative pro-algebra $\widehat{\op{Sym}_A}M^*=(\op{Sym}^{\leq k}_AM_i^*)_{(i,k)\in I\times\mathbb{N}}$. A differential
$$d_{(i,n)(j,m)}\in \op{Diff}_{p_{(i,n)(j,m)}}(\op{Sym}^{\leq n}_AM_i^*, \op{Sym}^{\leq m}_AM_j^*)\text{, for }i\geq j, n\geq m$$
is uniquely determined by its restrictions to $A$ and $M_i^*$, and hence admits a weight decomposition, with $f_{(i,n)(j,m)}^k(A)\in \op{Sym}^{ k}M_j^*$ (zero if $k>m$), and $f_{(i,n)(j,m)}^k(M_i^*)\in \op{Sym}^{k+1}M_j^*$  (zero if $k>m$). This induces a weight decomposition of a given differential $d_{\op{CE}}$ on $\widehat{\op{Sym}}_AM^*$. A differential $d_{\op{CE}}$ whose weight -1 component vanishes is equivalent to a SH LR structure on the pair $(A,M)$. Its weight zero component is equivalently a differential on $A$ together with a structure of a differential graded pro-$A$-module $M^*$, dually a structure of a differential graded dg $A$-module on $M$. 

Similarly, morphisms of differential graded pro-algebras
$$f:(\widehat{\op{Sym}_A}M^*,d_{\op{CE}})\to (\widehat{\op{Sym}_B}N^*,d_{\op{CE}})$$ admit weight decomposition. Morphisms of SH LR pairs are defined as opposite to such maps whose weight $-1$ component vanishes. In this case, the weight zero component is equivalently a map of cdga-s $A\to B$ together with a map of dg pro-finite $B$-modules\footnote{$B\otimes_A M^*:=(B\otimes_A M_i^*)_i $} $B\otimes_A M^*\to N^*$, dually, a morphism of dg $B$-modules $N\to B\otimes_A M$. We say that $f$ is a weak equivalence if both maps are quasi-isomorphisms. Observe that the notion of weak equivalences is homotopy coherent only if $B\otimes_A M$ is of correct homotopy type.

Take $A\in {\tt cdga}(k)$, and a free graded $A$-module $M={A}\langle m_i\rangle_{i\in I}$. Denoting by $\vec{I}$ the small category of finite subsets of $I$ and inclusions,
$$ {M}={A}\langle m_i\rangle_{i\in I}=\varinjlim_{J\in \vec{I}} {A}\langle m_j\rangle_{j\in J}.$$
Its dual pro-finite module is represented by the inverse system of surjections ${A}\langle m_j^*\rangle_{j\in J}\twoheadrightarrow {A}\langle m_{j'}^*\rangle_{j'\in J'}$, for $J'\subset J$, whose projective limit is the graded $A$ module $\prod_{i\in I} A\langle m_i^* \rangle$, denoted by ${A}\langle\langle m_i^*\rangle\rangle_{i\in I}$. Differential on the inverse system induces a differential on the projective limit.  If in the limit,
\begin{equation}\label{condition}dm_i^*=\sum_j a^j_i m_j^*,\end{equation}
then for a fixed $j$ there are at most finitely many indices $i$ with $a^i_j\neq 0$.
 Conversely, given a family $(a^i_j:i,j\in I)$ which satisfies the same finiteness condition, there exists a unique differential on ${ A}\langle\langle m_i^*\rangle\rangle_{i\in I}$ which lifts to  a differential of graded pro-modules and satisfies Equation (\ref{condition}). The differential in the limit squares to zero if and only if its lift does. The associated dg pro-$A$-module is dual to a cell complex of $A$-modules if and only if the differential satisfies rising condition $a^i_j=0$ for $i\leq j$. An important observation is that when $M$ is itself a cell complex in $A$-modules, the projective limit of its dual together with the induced differential is equal to the naive dual $\underline{\op{Hom}}_{{\tt Mod}(A)}(M,A)$. 

The same holds for morphisms. A morphism of graded pro-$A$-modules $f^*:{A}\langle\langle m_i^*\rangle\rangle_{i\in I}\to { A}\langle\langle n_j^*\rangle\rangle_{j\in J}$ induces in the limit a map of graded $A$-modules. If in the limit,
\begin{equation}\label{conditionm}f^*(m_i^*)=\sum_j a^j_i n_j^*,\end{equation}
then for a fixed $j$ there are at most finitely many indices $i$ with $a^i_j\neq 0$.  Conversely, given a family $(a^j_i:i\in J,j\in I)$ which satisfies the same finiteness condition, there exists a unique morphism in the limit, which lifts to a map of graded pro-modules and satisfies Equation (\ref{condition}). The map in the limit respects the differential if and only if its lift does.

Projective limit of the inverse system of graded-commutative algebras $\widehat{\op{Sym}}_{A}({M^*})$
is denoted by ${ A}[[m_i^*]]_{i\in I}$. Its elements are formal power series
$$y=\sum_{\substack{n\in \mathbb{N}\\i_1\ldots i_n\in I}}a^{i_1\ldots i_n}m_{i_1}^*\cdots m_{i_n}^*.$$
A differential on ${ A}[[m_i^*]]_{i\in I}$ which lifts to the level of graded pro-algebras  is uniquely determined by
$$d_{M^*}|_A\in\op{Der}(A,A[[m_i^*]]),\hspace{10pt} d_{M^*}m_i^*=\sum_{\substack{n\in \mathbb{N}_{>0}\\i_1\ldots i_n\in I}}a_i^{i_1\ldots i_n}m_{i_1}^*\cdots m_{i_n}^*,$$
where for any n-tuple $(i_1\ldots i_n)\in I^{\times n}$ there is at most finitely many $i\in I$ for which the coefficient $a_i^{i_1\ldots i_n}$ is non-zero. Differential squares to zero if and only if its lifts does.

Similarly, a morphism $f^*:{ A}[[m_i^*]]_{i\in I}\to{ B}[[n_j^*]]_{j\in J}$ which lifts to the level of graded pro-algebras is uniquely determined by
$$f^*|_{ A}\in\op{Hom}_{{ Mod}({ A})}({ A},{B}[[n_j^*]]),\hspace{10pt} f^*(m_i^*)=\sum_{\substack{n\in \mathbb{N}_{>0}\\j_1\ldots j_n\in J}}b_i^{j_1\ldots j_n}n_{j_1}^*\cdots n_{j_n}^*,$$
where for any n-tuple $(j_1\ldots j_n)\in J^{\times n}$ there is at most finitely many $i\in I$ for which the coefficient $b_i^{j_1\ldots j_n}$ is non-zero. It respects the differential if and only if its lift does.

\vfill
\Addresses

\end{document}